\documentclass[letterpaper,10 pt,conference]{ieeeconf}
\IEEEoverridecommandlockouts

\let\labelindent\relax

\usepackage{times}
\usepackage{xspace}
\usepackage[dvipsnames]{xcolor}

\usepackage{amsmath,amssymb,mathtools,nccmath,stmaryrd}

\usepackage{amsthm}
\usepackage{bm}
\usepackage{bbold,dsfont,bbm}

\usepackage{graphicx}
\usepackage{tikz,tikz-cd,pgfplots}
\pgfplotsset{compat=1.17}
\usepackage{adjustbox}

\usepackage[hidelinks]{hyperref}
\usepackage[capitalise,noabbrev]{cleveref}
\usepackage{cite}
\usepackage{enumitem}
\usepackage{multirow}
\usepackage{booktabs}
\usepackage{algorithm}
\usepackage{algpseudocode}
\usepackage[acronym]{glossaries-extra}
\setabbreviationstyle[acronym]{long-short}
\glssetcategoryattribute{acronym}{nohyperfirst}{true}

\usepackage{comment}

\usepackage{units}
\usepackage[per-mode=symbol]{siunitx}[=v2]
\DeclareSIUnit{\eur}{\euro}
\DeclareSIUnit{\usd}{USD}
\DeclareSIUnit{\mph}{mph}
\DeclareSIUnit{\month}{month}
\DeclareSIUnit{\year}{year}
\DeclareSIUnit{\million}{Mil}
\DeclareSIUnit{\mile}{mile}
\DeclareSIUnit{\car}{car}
\DeclareSIUnit{\train}{train}
\DeclareSIUnit{\mmveh}{\text{$\mu$}MV}
\DeclareSIUnit{\nounit}{-}
\usepackage[font=footnotesize]{caption}
\usepackage{subcaption}

\usepackage{array}
\newcommand{\PreserveBackslash}[1]{\let\temp=\\#1\let\\=\temp}
\newcolumntype{C}[1]{>{\PreserveBackslash\centering}p{#1}}
\newcolumntype{R}[1]{>{\PreserveBackslash\raggedleft}p{#1}}
\newcolumntype{L}[1]{>{\PreserveBackslash\raggedright}p{#1}}

\definecolor{lightblue}{rgb}{0.60784,0.76078,0.90196}
\definecolor{darkblue}{rgb}{0.26667,0.44706,0.76863}
\definecolor{lightgreen}{rgb}{0.66275,0.81569,0.55686}
\definecolor{darkgreen}{rgb}{0.43922,0.67843,0.27843}
\definecolor{orange}{rgb}{0.92941,0.49020,0.19216}
\definecolor{yellow}{rgb}{1.00000,0.75294,0.00000}
\definecolor{grey}{rgb}{0.64706,0.64706,0.64706}
\definecolor{purple}{rgb}{0.51373,0.23529,0.04706}
\newacronym{abk:cpo}{CPO}{complete partial order}
\newacronym{abk:cdp}{CDP}{co-design problem}
\newacronym{abk:ldp}{LDP}{linear design problem}
\newacronym{abk:lcdp}{LCDP}{linear co-design problem}
\newacronym{abk:cdpi}{CDPI}{co-design problem with implementation}
\newacronym{abk:dp}{DP}{design problem}
\newacronym{abk:dpi}{DPI}{design problem with implementation}
\newacronym{abk:mdpi}{MDPI}{monotone design problem with implementation}
\newacronym{abk:dcpo}{DCPO}{directed complete partial order}
\newacronym{abk:fme}{FME}{Fourier--Motzkin elimination}
\newacronym{abk:mcdp}{MCDP}{Monotone Co-Design Problem}
\newacronym{abk:mcmc}{MCMC}{Markov chain Monte-Carlo}
\newacronym{abk:molp}{MOLP}{multi-objective linear programming}
\newacronym{abk:vlp}{VLP}{vector linear programming}
\newacronym{abk:poset}{poset}{partially ordered set}
\newacronym{abk:uav}{UAV}{unmanned aerial vehicle}
\newacronym{abk:vi}{VI}{variational inference}
\newcommand{\F}[1]{\textcolor{dpgreen}{#1}}
\newcommand{\FI}[1]{\F{\textit{#1}}}
\newcommand{\funPosetF}{{\F{F}}}

\definecolor{dpgreen}{rgb}{0.0, 0.5, 0.0}

\definecolor{dpred}{rgb}{0.7, 0.0, 0.0}
\newcommand{\R}[1]{\textcolor{dpred}{#1}}
\newcommand{\RI}[1]{\R{\textit{#1}}}
\newcommand{\resPosetR}{{\R{R}}}

\ifPDFTeX

\else

\fi
\newcommand{\defeq}{\mathrel{\raisebox{-0.3ex}{$\overset{\text{\tiny def}}{=}$}}}
\newcommand{\setWithArg}[2]{\{#1 \mid #2\}}

\newcommand{\USet}{\mathtt{U}}
\newcommand{\USetOf}[1]{\USet(#1)}
\newcommand{\upperClosure}[1]{\uparrow\!#1}

\newcommand{\DP}{\mathsf{DP}}
\newcommand{\LDP}{\mathsf{LDP}}
\newcommand{\dpOf}[2]{\DP\{#1, #2\}}
\newcommand{\ldpOf}[2]{\LDP\{#1, #2\}}
\newcommand{\dprb}{{\mathrm{dp}}}

\newcommand{\ldprb}{{\mathrm{ldp}}}
\newcommand{\dprbOf}[1]{\dprb_#1}

\newcommand{\dprba}{\dprbOf{a}}
\newcommand{\dprbb}{\dprbOf{b}}

\newcommand{\posetproduct}{\times}

\newcommand{\mthen}{\fatsemi}
\newcommand{\stack}{\otimes}
\newcommand{\trace}{\text{Tr}}
\newcommand{\traceOf}[1]{\trace(#1)}

\newcommand{\op}{^{\mathrm{op}}}

\newcommand{\subsetXP}{X_\posetP}

\newcommand{\posetP}{\mathcal{P}}
\newcommand{\posetQ}{\mathcal{Q}}
\newcommand{\posetR}{\mathcal{R}}
\newcommand{\posetPprime}{\mathcal{P}^{\prime}}
\newcommand{\posetQprime}{\mathcal{Q}^{\prime}}

\newcommand{\tup}[1]{\left\langle #1 \right\rangle}

\newcommand{\posetleq}{\preceq}
\newcommand{\posetgeq}{\succeq}
\newcommand{\posceq}{\succeq}

\newcommand{\SetU}{U}

\newcommand{\poselxF}{x_\funPosetF}
\newcommand{\barposelxF}{\bar{x}_\funPosetF}

\newcommand{\poselxR}{x_\resPosetR}

\newcommand{\barposelxR}{\bar{x}_\resPosetR}
\newcommand{\poselxP}{x_\posetP}
\newcommand{\poselyP}{y_\posetP}
\newcommand{\poselxQ}{x_\posetQ}
\newcommand{\poselyQ}{y_\posetQ}
\newcommand{\poselxPprime}{x_{\posetPprime}}

\newcommand{\poselxQprime}{x_{\posetQprime}}

\newcommand{\poselxRnoRes}{x_\posetR}
\newcommand{\posx}{x}

\definecolor{baiocchi}{RGB}{193,221,245}

\usetikzlibrary{positioning,intersections, hobby, patterns, calc, decorations.pathmorphing, decorations.markings, shadows,shapes, cd,arrows.meta, fit,quotes}

\tikzset{
   tick/.style={postaction={
      decorate,
      decoration={markings, mark=at position 0.5 with {\draw[-] (0,.4ex) -- (0,-.4ex);}}}
   }
}
\tikzstyle{block} = [draw, rectangle, minimum height=2em, minimum width=3em,font=\bfseries,rounded corners,thick]
\tikzstyle{block} = [draw, rectangle, minimum height=2em, minimum width=3em]
\tikzstyle{block1} = [draw, rectangle, minimum height=1.5em, minimum width=2.5em]
\tikzstyle{blockDyn} = [draw, rectangle, minimum height=2.5em, minimum width=3.5em, align=center, inner sep=10pt, thick, fill=white, copy shadow={draw=black,fill=black,opacity=1,shadow xshift=0.5ex,shadow yshift=-0.5ex}]
\tikzstyle{blockAlg} = [draw, rectangle, minimum height=1.5em, minimum width=2.5em, align=center, inner sep=10pt, thick]
\tikzstyle{sum} = [draw,circle]

\tikzstyle{nodePre} = [circle, draw,inner sep=1pt,node contents={$\preceq$},thick]
\tikzstyle{nodePreEmpty} = [circle, draw,inner sep=1pt,thick]
\tikzstyle{nodePos} = [circle, draw,inner sep=1pt,node contents={$\posceq$},thick]
\tikzstyle{nodeProd} = [rectangle, draw,inner sep=4pt,node contents={$\times$},rounded corners,thick]
\tikzstyle{nodeSum} = [rectangle, draw,inner sep=4pt,node contents={$\mathbf{+}$},rounded corners,thick]

\definecolor{red}{rgb}{0.75, 0.0, 0.0}

\tikzset{fcname/.store in =\fcname, fcname={}}
\tikzset{funame/.store in =\funame, funame={}}
\tikzset{rcname/.store in =\rcname, rcname={}}
\tikzset{runame/.store in =\runame, runame={}}
\tikzset{whereres/.store in =\whereres, whereres=0.5}
\tikzset{wherefun/.store in =\wherefun, wherefun=0.5}
\tikzset{relres/.store in =\relres, relres={above}}
\tikzset{relfun/.store in =\relfun, relfun={above}}
\tikzset{posres/.store in =\posres, posres=1}
\tikzset{posfun/.store in =\posfun, posfun=1}
\tikzset{loos/.store in =\loos, loos=2}
\tikzset{feedback/.store in =\feedback, feedback=0}
\tikzset{
   DP/.style={
      label/.style={
         font=\everymath\expandafter{\the\everymath\scriptstyle},
         inner sep=5pt,
         node distance=2pt and -2pt},
      semithick,
      node distance=1 and 1,
      rconn/.style={color=white,opacity=0.0,postaction={decorate}, shorten <=3.2pt, shorten >= 0.8,
      decoration={markings, 
      mark= at position 0 with {
               \coordinate (a);
      },
      mark=at position .5 with
      {
              \ifthenelse{\equal{\feedback}{1}}{\def\angleOut{90}\def\angleIn{90}}{\def\angleOut{0}\def\angleIn{180}}    
              \coordinate (b);
              \draw[dashed,dpred,opacity=1.0] (a) to[out=\angleOut,in=\angleIn,looseness=\loos] 
              node[pos=\posres,\relres=\whereres mm,dpred,opacity=1,fill=white,inner sep=1pt,outer sep=1pt]{\footnotesize{\rcname}} (b);
      },
      mark= at position 1 with 
      {
             \ifthenelse{\equal{\feedback}{1}}{\def\angleOut{0}\def\angleIn{0}}{\def\angleOut{180}\def\angleIn{0}} 
              \ifthenelse{\equal{\feedback}{1}}{\def\symbol{\succeq}}{\def\symbol{\preceq}} 
              \coordinate (c);
              \draw[dpgreen,opacity=1.0] (c) to[out=\angleOut,in=\angleIn,looseness=\loos]
              node[pos=\posfun,\relfun=\wherefun mm,dpgreen,opacity=1,fill=white,inner sep=1pt,outer sep=1pt]{\footnotesize{\fcname}} (b){}; 
              \node[draw,circle,inner sep=0.5pt,color=black,fill=white,opacity=1.0] at (b) (nodepreceq) {$\symbol$}; 
      }
      }},
      runconn/.style={color=dpred,dashed,postaction={decorate},
      decoration={markings,
      mark= at position 1 with {
              \coordinate (a);
              \draw[dpred,opacity=1.0,dashed] ($(a)+(0.05,0)$) --++ (0.5,0) node[\relres,pos=\posres]{\footnotesize{\runame}};}
      }
      },
      funconn/.style={color=white,postaction={decorate},
      decoration={markings,
      mark= at position 0 with {
      \coordinate (a);
      \draw[dpgreen] ($(a)+(-0.05,0)$) -- ($(a)+(-0.5,0)$) node[\relfun, pos=\posfun]{\footnotesize{\funame}};}
      }
      },
      execute at begin picture={\tikzset{
         x=\dpx, y=\dpy,
         every fit/.style={inner xsep=\dpx, inner ysep=\dpy}}}
      },
   dpx/.store in=\dpx,
   dpx = 1.5cm,
   dpy/.store in=\dpy,
   dpy = 1.5ex,
   dp port sep/.store in=\dpportsep,
   dp port sep=2,
   dp port length/.store in=\dpportlen,
   dp port length=4pt,
   dp min width/.store in=\dpminwidth,
   dp min width=0.5cm,
   dp rounded corners/.store in=\dpcorners,
   dp rounded corners=2pt,
   dp small/.style={dp port sep=1, dp port length=2.5pt, dpx=.4cm, dp min width=.4cm, dpy=.7ex},
   dp/.code 2 args={
      \pgfmathsetlengthmacro{\dpheight}{\dpportsep * (max(#1,#2)) * \dpy}
      \pgfkeysalso{draw,%
        minimum width=\dpminwidth,%
        minimum height=\dpheight,%
        font=\bfseries,
        outer sep=0pt,%
        inner sep=5pt,%
        rounded corners=\dpcorners,
        thick,
        prefix after command={\pgfextra{\let\fixname\tikzlastnode}},
        append after command={\pgfextra{\draw
            \ifnum #1=0{} \else foreach \i in {1,...,#1} { 
            ($(\fixname.north west)!{\i/(#1+1)}!(\fixname.south west)$) +(0,0) node[solid,left,circle,color=dpgreen,draw,fill=dpgreen,scale=0.3] {} coordinate (\fixname_fun\i) -- +(0,0) coordinate (\fixname_fun\i')}\fi 
            \ifnum #2=0{} \else foreach \i in {1,...,#2} {
            ($(\fixname.north east)!{\i/(#2+1)}!(\fixname.south east)$) +(0,0) coordinate (\fixname_res\i') -- +(0,0) node[solid,right,circle,color=dpred,draw,fill=dpred,scale=0.3] {} coordinate (\fixname_res\i)}\fi;
         }}}
         },
      dp name/.style={append after command={\pgfextra{\node[label=center,inner sep=2pt,fill=white] at (\fixname) {\textbf{#1}};}}}
   }
\renewcommand{\baselinestretch}{1}

\theoremstyle{definition}

\newtheorem*{assumption*}{Assumption}
\newtheorem{theorem}{Theorem}

\newtheorem{lemma}[theorem]{Lemma}

\newtheorem{definition}{Definition}
\theoremstyle{remark}
\newtheorem*{remark}{Remark}

\Crefname{figure}{Fig.}{Figures}
\crefname{equation}{Eq.}{Eqs.}

\makeatletter
    \if@cref@capitalise
        \crefname{subsection}{Section}{Sections}
        \crefname{subsubsection}{Section}{Sections}
        \crefname{assumption}{Assumption}{Assumptions}
        \crefname{problem}{Problem}{Problems}
    \else
        \crefname{subsection}{section}{sections}
        \crefname{subsubsection}{section}{sections}
        \crefname{assumption}{assumption}{assumptions}
        \crefname{problem}{problem}{problems}
    \fi
\makeatother

\usetikzlibrary{positioning,intersections, hobby, patterns, calc, decorations.pathmorphing, decorations.markings, shadows,shapes, cd,arrows.meta, fit,quotes}
\tikzset{
   tick/.style={postaction={
      decorate,
      decoration={markings, mark=at position 0.5 with {\draw[-] (0,.4ex) -- (0,-.4ex);}}}
   }
}
\tikzstyle{block} = [draw, rectangle, minimum height=2em, minimum width=3em,font=\bfseries,rounded corners,thick]
\tikzstyle{block1} = [draw, rectangle, minimum height=1.5em, minimum width=2.5em]
\tikzstyle{blockDyn} = [draw, rectangle, minimum height=2.5em, minimum width=3.5em, align=center, inner sep=10pt, thick, fill=white, copy shadow={draw=black,fill=black,opacity=1,shadow xshift=0.5ex,shadow yshift=-0.5ex}]
\tikzstyle{blockAlg} = [draw, rectangle, minimum height=1.5em, minimum width=2.5em, align=center, inner sep=10pt, thick]
\tikzstyle{sum} = [draw,circle]

\tikzstyle{blockfill} = [block,rounded corners=4,fill=white]

\tikzstyle{nodePre} = [circle, draw,inner sep=1pt,node contents={$\preceq$},thick]
\tikzstyle{nodePreEmpty} = [circle, draw,inner sep=1pt,thick]
\tikzstyle{nodePos} = [circle, draw,inner sep=1pt,node contents={$\posceq$},thick]
\tikzstyle{nodeProd} = [rectangle, draw,inner sep=4pt,node contents={$\times$},rounded corners,thick]
\tikzstyle{nodeSum} = [rectangle, draw,inner sep=4pt,node contents={$\mathbf{+}$},rounded corners,thick]

\definecolor{DPgreen}{rgb}{0.0, 0.5, 0.0}
\definecolor{red}{rgb}{0.75, 0.0, 0.0}

\newif\ifmargincomments 
\margincommentstrue

\newif\ifextendedversion 

\begin{document}

\title{
\textbf{Scalable Co-Design via Linear Design Problems:\\ Compositional Theory and Algorithms}
}
\author{Yubo Cai, Yujun Huang, Meshal Alharbi, Gioele Zardini
\thanks{The authors are with the Laboratory for Information and Decision Systems, Massachusetts Institute of Technology, 02139 Cambridge (MA), USA, {\tt \{yubocai,yujun233,meshal,gzardini\}@mit.edu}}
\thanks{This material is based upon work supported by the Defense Advanced Research Projects Agency (DARPA) under Award No. D25AC00373. The views and conclusions contained in this document are those of the authors and should not be interpreted as representing the official policies, either expressed or implied, of the U.S. Government.}
}

\maketitle

\begin{abstract}
Designing complex engineered systems requires managing tightly coupled trade-offs between subsystem capabilities and resource requirements.
Monotone co-design provides a compositional language for such problems, but its generality does not by itself reveal which problem classes admit exact and scalable computation.
This paper isolates such a class by introducing \glspl{abk:ldp}: design problems whose feasible functionality--resource relations are polyhedra over Euclidean posets.
We show that queries on LDPs reduce exactly to Multi-Objective Linear Programs (MOLPs), thereby connecting monotone co-design semantics with polyhedral multiobjective optimization.
We further prove that LDPs are closed under the fundamental co-design interconnections, implying that any interconnection of linear components induces a system-level LDP.
To compute the resulting feasible sets, we develop two complementary constructions: a monolithic lifted formulation that preserves block-angular sparsity, and a compositional formulation that incrementally eliminates internal variables through polyhedral projection.
Beyond the exact linear setting, we show that convex co-design resource queries admit arbitrarily accurate polyhedral outer approximations, with recession-cone error identically zero for standard nonnegative resource cones.
Numerical studies on synthetic series-chain benchmarks, a gripper, and a rover co-design validate the theory.
\end{abstract}

\glsresetall


\section{Introduction}\label{sec:intro}
Designing complex engineered systems requires navigating tightly coupled trade-offs across heterogeneous subsystems.
In autonomy, robotics, and cyber-physical systems, choices in sensing, computation, actuation, energy, and physical structure jointly determine achievable performance, cost, and safety margins.
As a result, optimizing each subsystem in isolation can yield architectures that are globally infeasible or far from Pareto optimal.
This motivates \emph{co-design}: the joint design of interacting subsystems under coupled \F{functionality} and \R{resource} constraints~\cite{zardini2023co,MartinsLambe2013,Seshia2017TCAD}.

A principled foundation for this viewpoint is \emph{monotone co-design}, where each component is modeled as a \emph{\gls{abk:dp}}: a monotone relation between \glspl{abk:poset} of \F{functionalities} and \R{resources}~\cite{zardini2023co,censi2022,censi2015mathematical}.
Its key advantage is the separation of \emph{structure} from \emph{semantics}: the architecture is encoded by an interconnection graph, while each component contributes only its local feasibility relation.
Series, parallel, intersection, and feedback then induce a system-level \gls{abk:dp} compositionally from the component \glspl{abk:dp}, enabling modular design and reuse across domains~\cite{zardiniecc21,zardiniTaskdrivenModularCodesign2022,milojevic2025codei,neumann2024strategic,zardini2022co}.

The generality of monotone co-design is both its strength and its limitation.
It provides a domain-agnostic semantics for compositional design, but does not by itself identify which structured subclasses admit exact and scalable computation.
In many engineering applications, functionality--resource trade-offs are naturally linear or polyhedral, yet current co-design frameworks do not isolate this case as a tractable subclass with its own closure theory, query semantics, and algorithms.
This paper addresses that gap by introducing \emph{\glspl{abk:ldp}}, whose feasible functionality--resource sets are polyhedra over Euclidean \glspl{abk:poset}.
Within monotone co-design, \glspl{abk:ldp} play a role analogous to linear programming in mathematical optimization: they isolate a tractable structural core, admit exact algorithmic treatment, and provide a natural backbone for approximating broader convex models.

\subsection{Related Work}
This work sits at the intersection of three lines of research.
First, system-level design and multidisciplinary design optimization emphasize that engineering decisions must be coordinated across interacting disciplines and components~\cite{MartinsLambe2013,Seshia2017TCAD,karkus2023diffstack}.
Second, monotone co-design provides a compositional and order-theoretic language for such problems, with exact semantics based on monotone relations and fixed-point computation~\cite{censi2015mathematical,censi2017uncertainty,censi2022}.
Third, our computational viewpoint connects co-design to multi-objective and vector linear optimization, where efficient-frontier computation, polyhedral projection, and Benson-type algorithms are well developed~\cite{ehrgott2005multicriteria,Lohne2011,lohne_equivalence_2016,lohne_vector_2017,benson1998outer}.

Existing co-design tools are intentionally general.
For instance, \textsc{MCDPL} can represent broad classes of monotone relations and handle continuous specifications through discretization and consistent relaxations~\cite{censi2015mathematical,zardini2023co}.
However, it does not exploit linear or polyhedral structure explicitly; thus even simple linear trade-offs may be treated through generic antichain propagation or discretization rather than exact polyhedral algorithms.
What is missing is a structural linear theory \emph{within} the co-design formalism itself.

\subsection{Statement of Contribution}
The contributions of the paper are as follows.
First, we define \glspl{abk:ldp} and show that their natural queries reduce to \glspl{abk:molp}/\glspl{abk:vlp}, establishing a direct bridge between monotone co-design and polyhedral multi-objective optimization.
Second, we prove that \glspl{abk:ldp} are closed under the main co-design interconnections; hence any \gls{abk:lcdp} assembled from linear components induces a system-level \gls{abk:ldp}.
Third, we develop two complementary feasible-set construction methods for \glspl{abk:lcdp}: (i) a \emph{monolithic} lifted formulation that preserves block-angular sparsity, and (ii) a \emph{compositional} formulation that eliminates internal variables through graph-guided projection.
Fourth, we extend the framework beyond the exact linear case by showing that convex co-design resource queries admit arbitrarily accurate polyhedral outer approximations; for the standard resource cone $\mathbb{R}_{\ge 0}^{n}$, the recession-cone error is structurally exact ($\delta=0$), so accuracy is governed solely by the bounded-part error~$\varepsilon$.
Finally, we validate the theory on three benchmarks: a synthetic series-chain family that isolates the trade-off between sparsity-preserving and elimination-based constructions, a rigid gripper \gls{abk:ldp} illustrating the exact polyhedral regime, and a nonlinear rover case study in which tangent-plane outer approximations yield accurate $\left(\varepsilon,0\right)$ surrogates and substantial runtime gains over \textsc{MCDPL}.
\section{Linear Co-Design Theory}\label{sec:linear-dp-theory}

We formalize \glspl{abk:ldp}, i.e., \glspl{abk:dp} whose feasible sets are polyhedra over Euclidean \glspl{abk:poset}.
We first recall the order-theoretic and convex-geometric notation used throughout, then introduce \glspl{abk:ldp} and their query semantics.

\subsection{Mathematical preliminaries}\label{subsec:math_prelim}
\subsubsection{Background on orders}\label{sec:app_order}

\begin{definition}[Poset]
A \emph{\gls{abk:poset}} is a tuple~$\posetP =\tup{P,\preceq_\posetP}$, where $P$ is a set and~$\preceq_\posetP$ is a partial order (a reflexive, transitive, and antisymmetric relation).
If clear from context, we use~$\posetP$ for a \gls{abk:poset}, and~$\posetleq$ for its order.
\end{definition}

\begin{definition}[Opposite poset]
The \emph{opposite} of a \gls{abk:poset}~$\posetP = \tup{P,\preceq_\posetP}$ is the poset $\posetP\op \defeq\tup{P,\posetleq_{\posetP}\op}$ with the same elements and reversed ordering:
    $
        \poselxP \posetleq_{\posetP}\op \poselyP \Leftrightarrow
        \poselyP \posetleq_{\posetP}\poselxP
    $.
\end{definition}

\begin{definition}[Product poset]
    Given \glspl{abk:poset} $\tup{P,\preceq_{\posetP}}$ and $\tup{Q,\preceq_{\posetQ}}$, their \emph{product} $\tup{P\times Q,\preceq_{\posetP\times \posetQ}}$ is the poset with
    \begin{equation*}
        \tup{\poselxP,\poselxQ}\preceq_{\posetP\times \posetQ}\tup{\poselyP,\poselyQ} \Leftrightarrow (\poselxP \preceq_{\posetP} \poselyP) \wedge (\poselxQ \preceq_\posetQ \poselyQ).
    \end{equation*}
\end{definition}

\begin{definition}[Upper closure]\label{def:upper-closure}
Let $\posetP$ be a \gls{abk:poset}. 
The \emph{upper closure} of a subset $\subsetXP \subseteq \posetP$ contains all elements of $\posetP$ that are greater or equal to some $\poselyP \in \subsetXP$:
    \begin{equation*}
        \upperClosure{\subsetXP} \defeq \setWithArg{\poselxP \in \posetP}{\exists \poselyP \in \subsetXP : \poselyP \posetleq_\posetP \poselxP}.
    \end{equation*}
\end{definition}

\begin{definition}[Upper set]\label{def:uppersets-of-posets}
    A subset $\subsetXP \subseteq \posetP$ of a \gls{abk:poset} is called an \emph{upper set} if it is upwards closed: $\upperClosure{\subsetXP} = \subsetXP$.
    We write $\USetOf{\posetP}$ for the set of upper sets of $\posetP$.
    We regard $\USetOf{\posetP}$ as a \gls{abk:poset} with $\SetU \posetleq \SetU' \Leftrightarrow \SetU \supseteq \SetU'$.
    Thus larger upper sets are considered \emph{smaller} in this order (encoding weaker constraints).
\end{definition}

\subsubsection{Convex geometry}\label{sec:cvx_def}

Throughout,~$\Vert\cdot\Vert$ denotes the Euclidean norm.

\begin{definition}[Set operations]\label{def:set-ops}
For a set $C\subseteq\mathbb{R}^n$, define the \emph{convex hull}, \emph{conical hull}, and \emph{closure} by
\begin{align*}
  \operatorname{conv}(C)
  &\defeq
  \Bigl\{\textstyle\sum_{i=1}^k \lambda_i c_i \mid
  k\in\mathbb{N}, c_i\in C, \lambda_i\ge 0, \sum_i\lambda_i=1\Bigr\},\\
  \operatorname{cone}(C)
  &\defeq
  \Bigl\{\textstyle\sum_{i=1}^k \lambda_i c_i \mid
  k\in\mathbb{N},\ c_i\in C,\ \lambda_i\ge 0\Bigr\},\\
  \operatorname{cl}(C)
  &\defeq
  \bigl\{x\in\mathbb{R}^n \mid \exists\,(x_k)_{k\ge 1}\subseteq C,\ x_k\to x\bigr\}.
\end{align*}
These are the smallest convex set, convex cone, and closed set containing $C$, respectively.
A point $v\in P$ is a \emph{vertex} (extreme point) of a set $P$ if it cannot be expressed as $v=\tfrac{1}{2}(x+y)$ for distinct $x,y\in P$.
We write $\operatorname{vert}(P)$ for the set of vertices of a polyhedron $P$.
\end{definition}

\begin{definition}[Basic convex sets]
    The following sets are basic examples of convex sets in $\mathbb{R}^n$:
    \begin{enumerate}[label=(\arabic*), leftmargin=*]
        \item \emph{Linear subspace:} $L \defeq \setWithArg{x \in \mathbb{R}^n}{Ax = 0}$ for some matrix $A \in \mathbb{R}^{m \times n}$.
        \item \emph{Affine subspace:} $M \defeq L + x_0$, where $L$ is a linear subspace and $x_0 \in \mathbb{R}^n$. Equivalently, $M = \setWithArg{x \in \mathbb{R}^n}{Ax=b}$ with $b = A x_0$.
        \item \emph{Hyperplane:} $H \defeq \setWithArg{x \in \mathbb{R}^n}{a^\top x = b}$, where $a \in \mathbb{R}^n$ and $b \in \mathbb{R}$. It is an $(n-1)$-dimensional affine subspace.
        \item \emph{Halfspace:} $H^{+} \defeq \setWithArg{x \in \mathbb{R}^n}{a^\top x \ge b}$ and $H^{-} \defeq \setWithArg{x \in \mathbb{R}^n}{a^\top x \le b}$, which are (by default) closed halfspaces determined by the hyperplane $H$.
    \end{enumerate}
\end{definition}
\begin{definition}[Polyhedron]
    A \emph{polyhedron} $P$ is the intersection of a finite number of halfspaces. That is:
    \begin{equation*}
    \begin{aligned}
        P= \bigcap_{i=1}^m H_i^{+}\defeq\bigcap_{i=1}^m\left\{x \in \mathbb{R}^n: a_i^{\top} x \geq b_i\right\}=\left\{x \in \mathbb{R}^n: A x \geq b\right\},
    \end{aligned}
    \end{equation*}
    where~$a_i^{\top}$ form the rows of $A$ and $b_i$ are the components of $b$. This is called the \emph{$\mathbf{H}$-representation} of a polyhedron.
    A polyhedron can also be written in \emph{$\mathbf{V}$-representation} as
    \begin{equation*}
        P=\operatorname{conv}\{v^1,\ldots,v^p\}+\operatorname{cone}\{r^1,\ldots,r^q\},
    \end{equation*}
    where $\{v^1,\ldots,v^p\}$ are vertices and $\{r^1,\ldots,r^q\}$ are extreme rays.
    A bounded polyhedron ($q=0$) is called a \emph{polytope}.
\end{definition}

\begin{remark}[H/V representations and conversion]
For polyhedra in finite-dimensional spaces, $\mathbf{H}$- and $\mathbf{V}$-representations are equivalent (Minkowski--Weyl theorem).
Converting $\mathbf{H}\to\mathbf{V}$ amounts to vertex/ray enumeration, while converting $\mathbf{V}\to\mathbf{H}$ amounts to facet enumeration; both can be exponential in the worst case.
In this paper, feasible sets are constructed primarily in $\mathbf{H}$-representation from linear constraints, and $\mathbf{V}$-information is recovered through polyhedral projection or \gls{abk:molp}/\gls{abk:vlp} solvers when needed~\cite{fukuda2016lecture,lohne_equivalence_2016,lohne_vector_2017}.
\end{remark}

\subsection{Background on co-design and queries}\label{subsec:co-design-intro}
Co-design provides a compositional perspective for analyzing complex systems by making explicit the trade-offs between what a subsystem can \F{provide} and what it must \R{consume}.
As a motivating example, consider the design of \glspl{abk:uav} (see \cite[Fig 1(a)]{censi2017uncertainty}), in which perception is a critical sub-module.
Holding the remaining system conditions fixed, the perception unit can be viewed as delivering a target level of \FI{detection accuracy} at the cost of \RI{computation power}, under prescribed \RI{weather conditions}.
Although accuracy and computation can be naturally represented by ordered real vectors, weather conditions are often only partially ordered (and may contain incomparable scenarios).
For example, a clear night and a foggy day induce qualitatively different sensing regimes, so a design tailored to one condition may not transfer to the other.
Monotone co-design formalizes these heterogeneous trade-offs by representing components as \glspl{abk:dp} defined over \glspl{abk:poset} of \F{functionalities} and \R{resources}.

\begin{definition}[\Gls{abk:dp}]\label{def:dp}
    Given \glspl{abk:poset}~$\funPosetF$ and $\resPosetR$ of \FI{functionalities} and \RI{resources}, a \emph{\gls{abk:dp}} is an \emph{upper set} of~$\funPosetF\op \posetproduct \resPosetR$.
    We denote the set of such \glspl{abk:dp} by~$\dpOf{\funPosetF}{\resPosetR}$.
    Given a \gls{abk:dp}~$\dprb$, a pair~$\tup{\poselxF, \poselxR}$ of functionality~$\poselxF$ and resource~$\poselxR$ is \emph{feasible} if~$\tup{\poselxF, \poselxR} \in \dprb$.
    We order~$\dpOf{\funPosetF}{\resPosetR}$ by inclusion:~$\dprba \posetleq \dprbb \Leftrightarrow \dprba \subseteq \dprbb$.
    Note that this is the opposite of the ordering used for upper sets, i.e., larger feasible sets correspond to ``better'' \glspl{abk:dp}.
\end{definition}

The upper set condition captures the following intuition: if a resource $\poselxR$ suffices to provide functionality $\poselxF$, then it also suffices for any worse functionality $\poselxF' \posetleq \poselxF$.
Moreover, any better resource $\poselxR' \posetgeq \poselxR$ should also suffice to provide $\poselxF$.

A core tenet of co-design is to compose systems out of simpler sub-systems. Such composites are formalized as multi-graphs of \glspl{abk:dp} called \emph{\glspl{abk:cdp}}.
We summarize the main composition operations in \cref{def:dp-interconnections}, with some shown diagrammatically in \cref{fig:dp_def}.

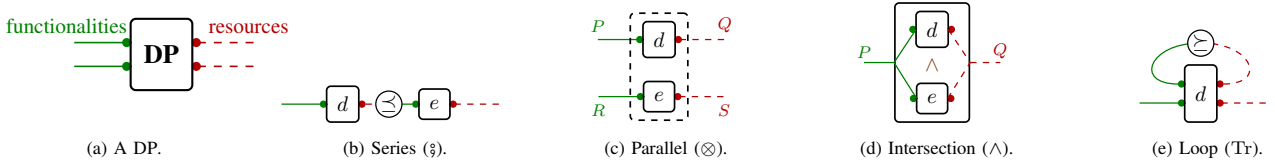
\begin{figure*}[tb]
    \centering
    \begin{subfigure}[b]{0.19\textwidth}
        \centering
        \begin{tikzpicture}[DP]
            \node[dp={2}{2}] (cnt) {DP};
            \draw[runconn, runame={resources}, relres=above,posres=0.9] (cnt_res1){};
            \draw[runconn, runame={}, relres=above,posres=0.9] (cnt_res2){};
            \draw[funconn, funame={functionalities},relfun=above,posfun=1.15] (cnt_fun1){};
            \draw[funconn, funame={},relfun=above,posfun=1.15] (cnt_fun2){};
\end{tikzpicture}
        \subcaption{A \gls{abk:dp}.\label{fig:mathcodesign}}
    \end{subfigure}\hfill
    \begin{subfigure}[b]{0.19\textwidth}
        \centering
        \scalebox{0.8}{\begin{tikzpicture}[DP]
    \node[dp={1}{1}] (f) {$d$};
    \node[dp={1}{1}, right=1cm of f] (g) {$e$};
    \draw[rconn, rcname={}, fcname={}] (f_res1)  to (g_fun1);
    \draw[runconn, runame={}] (g_res1);
    \draw[funconn, funame={}] (f_fun1);
\end{tikzpicture}}
        \subcaption{Series ($\mthen$).}
    \end{subfigure}\hfill
    \begin{subfigure}[b]{0.19\textwidth}
        \centering
        \scalebox{0.8}{\begin{tikzpicture}[DP]
    \node[dp={1}{1}] (d) {$d$};
    \node[dp={1}{1}, below=0.4cm of d] (e) {$e$};

    \draw[funconn, funame={$P$}, relfun=above] (d_fun1);
    \draw[runconn, runame={$Q$}, relres=above] (d_res1);
    \draw[funconn, funame={$R$}, relfun=below] (e_fun1);
    \draw[runconn, runame={$S$}, relres=below] (e_res1);

    \node[draw, dashed, rounded corners=2pt, thick,
          fit=(d)(e), inner xsep=6pt, inner ysep=4pt] (enc) {};
\end{tikzpicture}}
        \subcaption{Parallel ($\stack$).}
    \end{subfigure}\hfill
    \begin{subfigure}[b]{0.19\textwidth}
        \centering
        \scalebox{0.8}{\begin{tikzpicture}[DP]
    \node[dp={1}{1}] (d) {$d$};
    \node[dp={1}{1}, below=0.6cm of d] (e) {$e$};
    \node[draw, solid, rounded corners=2pt, thick,
          fit=(d)(e), inner xsep=10pt, inner ysep=4pt] (enc) {};
    \coordinate (split) at (enc.west);
    \draw[dpgreen] (split) -- (d_fun1);
    \draw[dpgreen] (split) -- (e_fun1);
    \draw[dpgreen] ($(split)+(-0.5cm,0)$) -- (split)
          node[above, pos=0, dpgreen, fill=white, inner sep=1pt, outer sep=1pt]
          {\footnotesize{$P$}};
    \coordinate (merge) at (enc.east);
    \draw[dpred, dashed] (d_res1) -- (merge);
    \draw[dpred, dashed] (e_res1) -- (merge);
    \draw[dpred, dashed] (merge) -- ($(merge)+(0.5cm,0)$)
          node[above, pos=1, dpred, fill=white, inner sep=1pt, outer sep=1pt]
          {\footnotesize{$Q$}};
    \node[text=purple] at ($(d.south)!0.5!(e.north)$) {$\wedge$};
\end{tikzpicture}}
        \subcaption{Intersection ($\wedge$).}
    \end{subfigure}\hfill
    \begin{subfigure}[b]{0.19\textwidth}
        \centering
        \scalebox{0.8}{\begin{tikzpicture}[DP]
    \node[dp={2}{2}] (f) {$d$};
    \draw[runconn, runame={}] (f_res2){};
    \draw[funconn, funame={}] (f_fun2){};
    \draw[rconn,rcname={},fcname={},feedback=1,loos=3] (f_res1) -- ($(f)+(0,4)$) |- (f_fun1);
\end{tikzpicture}}
        \subcaption{Loop ($\operatorname{Tr}$).}
    \end{subfigure}
    \caption{A \gls{abk:dp} is a monotone relation between posets of \F{functionalities} and \R{resources}~(a), and can be composed via various operations~(b)--(e).}
    \label{fig:dp_def}
    \vspace{-3mm}
\end{figure*}

\begin{definition}[Composition operations for \glspl{abk:dp}]\label{def:dp-interconnections}
    The following operations construct new \glspl{abk:dp} from old:
    
    \emph{Series}: Given \glspl{abk:dp} $\dprba \in \dpOf{\posetP}{\posetQ}$ and $\dprbb \in \dpOf{\posetQ}{\posetR}$, their series connection $\dprba \mthen \dprbb \in \dpOf{\posetP}{\posetR}$ is defined as
    \[
        \{\tup{\poselxP, \poselxRnoRes} \mid
        \exists \poselxQ : \tup{\poselxP, \poselxQ} \in \dprba
        \text{ and }\tup{\poselxQ, \poselxRnoRes} \in \dprbb\}.
    \]
    This models situations where $\dprba$ uses the functionalities provided by $\dprbb$ as its resources.
    
    \emph{Parallel}: For~$\dprba \in \dpOf{\posetP}{\posetQ}$,~$\dprba' \in \dpOf{\posetP'}{\posetQ'}$, their parallel connection~$\dprba \stack \dprba' \in \dpOf{\posetP \posetproduct \posetP'}{\posetQ \posetproduct \posetQ'}$ is
    \[
        \{\tup{\tup{\poselxP,\poselxP'}, \tup{\poselxQ,\poselxQ'}} \mid
        \tup{\poselxP, \poselxQ} \in \dprba
        , \tup{\poselxP', \poselxQ'} \in \dprba' \},
    \]
    representing two non-interacting systems.
    
    \emph{Intersection}: For $\dprba, \dprbb \in \ldpOf{\posetP}{\posetQ}$, their intersection $\dprba \wedge \dprbb \in \dpOf{\posetP}{\posetQ}$ is 
    \[
        \{ \tup{\poselxP, \poselxQ} \mid
        \tup{\poselxP, \poselxQ} \in \dprba,
        \tup{\poselxP, \poselxQ} \in \dprbb
        \}.
    \]
    representing a conjunction of constraints.
    
    \emph{Feedback/Trace}: For $\dprb \in \dpOf{\posetP \posetproduct \posetR}{\posetQ \posetproduct \posetR}$, its trace $\traceOf{\dprb} \in \dpOf{\posetP}{\posetQ}$ is defined as
    \[
        \{ \tup{\poselxP, \poselxQ} \mid
        \exists \poselxRnoRes : \tup{\tup{\poselxP,\poselxRnoRes}, \tup{\poselxQ,\poselxRnoRes}} \in \dprb \}.
    \]
    This models the case where functionalities provided by $\dprb$ are used as its own resources.
\end{definition}

These compositions allow the construction of a \gls{abk:dp} from \glspl{abk:cdp}, such as \cref{fig:gripper-architecture,fig:rover-codesign-architecture} below.

\begin{definition}[\glspl{abk:cdp}]\label{def:cdp}
    We call $\tup{\funPosetF, \resPosetR, \tup{\mathcal{V}, \mathcal{E}}}$ a \gls{abk:cdp}, where $\funPosetF = \Pi_{i = 1}^{n_{\funPosetF}} \funPosetF_i$ and $\resPosetR = \Pi_{j = 1}^{n_{\resPosetR}} \resPosetR_j$ are the system-level functionalities and resources, and $\tup{\mathcal{V}, \mathcal{E}}$ is a multi-graph where:
    \begin{enumerate}[label=(\roman*)]
        \item Each node in $\mathcal{V}$ is a \gls{abk:dp}~$d$ with functionality $\funPosetF_{d} = \Pi_{i = 1}^{n_{d, \funPosetF}} \funPosetF_{d,i}$ and resource $\resPosetR_d = \Pi_{j = 1}^{n_{d, \resPosetR}} \resPosetR_{d,j}$.
        \item Each edge is a pair $\tup{\tup{d_1, i_1}, \tup{d_2, j_2}}$ with the corresponding \glspl{abk:poset} $\resPosetR_{d_1,i_1}$ and $\funPosetF_{d_2, j_2}$ being the same, representing the inequality $\posx_{\resPosetR_{d_1, i_1}} \posetleq \posx_{\funPosetF{d_2, j_2}}$ (the $i_1$-th resource required by $d_1 \in \mathcal{V}$ is at most the $j_2$-th functionality provided by $d_2 \in \mathcal{V}$).
        \item The system-level functionality/resource is the product of the unconnected functionalities/resources.
    \end{enumerate}
\end{definition}

With these modeling primitives in hand, we can ask for optimal solutions to \glspl{abk:dp}, formalized as \emph{queries}.

\begin{definition}[Querying \glspl{abk:dp}]\label{def:queries-dp}
    Given a \gls{abk:dp}~$\dprb \in \dpOf{\funPosetF}{\resPosetR}$, we define two basic queries:
    \begin{enumerate}
        \item \emph{Fix \FI{functionalities} minimize \RI{resources}}: for a fixed $\poselxF \in \funPosetF$, return the upper set of resources $\poselxR \in \resPosetR$ that make $\tup{\poselxF, \poselxR}$ feasible with respect to $\dprb$. We can view this query as a monotone map $q \colon \funPosetF \to \USetOf{\resPosetR}$.
        \item \emph{Fix \RI{resources} maximize \FI{functionalities}}: For a fixed $\poselxR \in \resPosetR$, return the upper set of functionalities $\poselxF \in \funPosetF$ that make $\tup{\poselxF, \poselxR}$ feasible with respect to $\dprb$. We can view this query as a monotone map $q' \colon \resPosetR \to \USetOf{\funPosetF\op}$.
    \end{enumerate}
    In applications we are typically interested in the \emph{minimal} resources in $q(\poselxF)$ and the \emph{maximal} functionalities in $q'(\poselxR)$.
\end{definition}

When $\funPosetF$ and $\resPosetR$ are general \glspl{abk:poset}, there might be a \emph{Pareto front} of optimal values and solutions to the queries, instead of a unique one.
Using the compositional structure, one can derive efficient algorithms to calculate the query results for complex \glspl{abk:dp} defined by a multi-graph of sub-systems~\cite{zardini2023co}.

\subsection{Linear co-design problems and queries}\label{sec:ldp}

We next focus on \glspl{abk:dp} endowed with additional geometric structure.
Based on \cref{def:dp}, we introduce \glspl{abk:ldp}.

\begin{definition}[Euclidean poset and sub-poset]\label{def:euclidean-poset}\label{def:euclidean-subposet}
For $n \in \mathbb{N}$, the \emph{Euclidean poset} $\tup{\mathbb{R}_{\geq 0}^n, \preceq}$ is the set $\mathbb{R}_{\geq 0}^n \defeq \{ x \in \mathbb{R}^n \mid x_i \geq 0,\; i=1,\dots,n \}$ equipped with the componentwise order $x \preceq y \Leftrightarrow x_i \leq y_i$ for all $i$.
A subset $S \subseteq \mathbb{R}_{\geq 0}^n$ with the restricted order is called a \emph{Euclidean sub-poset}.
\end{definition}

\begin{definition}[\gls{abk:ldp}]\label{def:ldp}
Let~$\funPosetF \subseteq \langle \mathbb{R}_{\geq 0}^{n_\funPosetF}, \preceq\rangle$,~$\resPosetR \subseteq \langle\mathbb{R}_{\geq 0}^{n_\resPosetR}, \preceq\rangle$ be Euclidean sub-posets.
A \emph{\gls{abk:ldp}} is a \gls{abk:dp} $\dprb \in \dpOf{\funPosetF}{\resPosetR}$ for which the set of all feasible pairs $\tup{\poselxF, \poselxR}$ forms a polyhedron in the underlying Euclidean space.
Equivalently, there exist~$A_{\funPosetF} \in \mathbb{R}^{m \times n_{\funPosetF}}$,~$A_{\resPosetR} \in \mathbb{R}^{m \times n_{\resPosetR}}$,~$b \in \mathbb{R}^m$ such that
    \begin{equation}\label{eq:ldp-representation}
        \ldprb = \left\{ \tup{\poselxF, \poselxR} \in \funPosetF\op \times \resPosetR \;\Big|\; A_{\funPosetF} \poselxF + A_{\resPosetR} \poselxR \geq b \right\}.
    \end{equation}
    We denote the set of all such \glspl{abk:ldp} by $\ldpOf{\funPosetF}{\resPosetR}$. For conciseness, we can write the condition as $A_{\ldprb} x_{\ldprb} \geq b$, where $A_{\ldprb}=\left[A_{\funPosetF}\;\middle|\; A_{\resPosetR}\right]$ and $x_{\ldprb} = \left[\poselxF^{\top}, \poselxR^{\top} \right]^{\top}$.
\end{definition}

\begin{remark}
Since $\funPosetF$ and $\funPosetF\op$ share the same underlying set, \cref{eq:ldp-representation} is written over $\funPosetF \times \resPosetR$ and the order reversal appears only in the requirement that $\dprb$ be an upper set in $\funPosetF\op \times \resPosetR$.
Not every polyhedron in $\funPosetF \times \resPosetR$ therefore defines an \gls{abk:ldp}: we additionally require the monotonicity encoded by \cref{def:dp}.
\end{remark}

Following the query formalism in \cref{def:queries-dp} (see~\cite{zardini2023co}), the \emph{fix \FI{functionalities} minimize \RI{resources}} query asks, for a fixed $\poselxF \in \funPosetF$, for the set of \emph{minimal} (Pareto) resources $\poselxR$ such that $\tup{\poselxF, \poselxR}$ is feasible; the \emph{fix \RI{resources} maximize \FI{functionalities}} query asks, for a fixed $\poselxR \in \resPosetR$, for the set of \emph{maximal} (Pareto) functionalities $\poselxF$ such that $\tup{\poselxF, \poselxR}$ is feasible.
For \glspl{abk:ldp}, these become \glspl{abk:molp}.

\begin{definition}[\gls{abk:molp}]\label{def:molp}
    A \gls{abk:molp} in standard form is the optimization problem
    \begin{equation*}
        \mathrm{Min}_{\preceq} \quad C x \quad \text{subject to} \quad A x \geq b,\quad x \geq 0,
    \end{equation*}
    where $C \in \mathbb{R}^{q \times n}$, $A \in \mathbb{R}^{m \times n}$, $b \in \mathbb{R}^m$, the variable $x \in \mathbb{R}^n$, and $\preceq$ is the componentwise partial order on $\mathbb{R}^q$.
    We call the set of feasible points $X := \{ x \geq 0 : A x \geq b \}$ the \emph{feasible set} of the \gls{abk:molp}.
\end{definition}

\begin{definition}[Efficient point]\label{def:efficient-point}
    A feasible point $x^* \in X$ is called \emph{efficient} (or \emph{Pareto optimal}) for the \gls{abk:molp} in \cref{def:molp} if there is no feasible $x \in X$ with $Cx \preceq Cx^*$ and $Cx \neq Cx^*$.
    Equivalently, $Cx^* \notin CX + \mathbb{R}_{\ge 0}^q \setminus \{0\}$, where $CX := \{Cx \mid x\in X\}$.
    The set of all efficient points is denoted $X_E$, and its image $CX_E$ is the \emph{efficient frontier} (or \emph{Pareto front}).
\end{definition}

\begin{definition}[Upper image of a \gls{abk:molp}]\label{def:upper-image-molp}
    For the \gls{abk:molp} in \cref{def:molp}, the \emph{upper image} is~$\posetP \defeq \operatorname{cl}\big(CX + \mathbb{R}_{\ge 0}^q\big)$.
    In the polyhedral linear case here, $\posetP$ is a polyhedron whose lower boundary coincides with the efficient frontier~$CX_E$.
\end{definition}

\begin{remark}[Vector linear program]
    A \emph{vector linear program} (VLP) generalizes \gls{abk:molp} by replacing the componentwise ordering cone $\mathbb{R}_{\ge 0}^q$ with an arbitrary polyhedral cone $K \subseteq \mathbb{R}^q$~\cite{lohne_vector_2017}:
     \[
        \mathrm{Min}_{K}\, Cx \quad \text{subject to} \quad Ax \ge b,\quad x \ge 0,
    \]
    where $y \preceq_K z \Leftrightarrow z - y \in K$.
    When $K = \mathbb{R}_{\ge 0}^q$, the VLP reduces to the standard \gls{abk:molp}.
    The Bensolve tools and solver~\cite{ciripoi2019calculus,lohne_vector_2017} computes $\mathbf{V}$- and $\mathbf{H}$-representations of the upper image $\posetP$ for VLPs.
\end{remark}

\begin{lemma}[Queries of \gls{abk:ldp} are equivalent to \gls{abk:molp}]\label{lem:ldp-queries-molp}
Consider an \gls{abk:ldp} $\dprb \in \ldpOf{\funPosetF}{\resPosetR}$ with feasible set
    $\ldprb = \{ \tup{\poselxF, \poselxR} \in \funPosetF \times \resPosetR : A_{\funPosetF} \poselxF + A_{\resPosetR} \poselxR \geq b \}$.
    \begin{enumerate}[label=(\roman*)]
        \item \emph{Fix \FI{functionalities} minimize \RI{resources}}: For fixed $\barposelxF \in \funPosetF$, the set of minimal resources that make $\tup{\barposelxF, \poselxR}$ feasible coincides with the set of efficient solutions of the \gls{abk:molp}
              \[
                  \mathrm{Min}_{\preceq} \; \poselxR \quad \text{s.t.} \; A_{\resPosetR} \poselxR \geq b - A_{\funPosetF} \barposelxF,\quad \poselxR \geq 0.
              \]
        \item \emph{Fix \RI{resources} maximize \FI{functionalities}}: For fixed $\barposelxR \in \resPosetR$, the set of maximal functionalities that make $\tup{\poselxF, \barposelxR}$ feasible coincides with the set of efficient solutions of the \gls{abk:molp}
              \[
                  \mathrm{Max}_{\preceq} \; \poselxF \quad \text{s.t.} \; A_{\funPosetF} \poselxF \geq b - A_{\resPosetR} \barposelxR,\quad \poselxF \geq 0,
              \]
              which is equivalent to $\mathrm{Min}_{\preceq} \; (-\poselxF)$ subject to the same constraints (efficient points are unchanged).
        \item Conversely, every \gls{abk:molp} of the form $\mathrm{Min}_{\preceq}\, x$ subject to $A x \geq b$, $x \geq 0$ is the fix-\FI{functionalities}-minimize-\RI{resources} query of the \gls{abk:ldp} with functionality space $\funPosetF = \{0\} \subseteq \mathbb{R}_{\geq 0}$ (singleton), resource space $\resPosetR = \mathbb{R}_{\geq 0}^n$, and feasible set $\ldprb = \{ \tup{0, x} : A x \geq b,\; x \geq 0 \}$ (i.e., $A_{\funPosetF} = 0$, $A_{\resPosetR} = A$).
    \end{enumerate}
\end{lemma}

\begin{proof}
We prove the three statements in order.
    \textbf{(i)} Fix $\barposelxF \in \funPosetF$ and consider the associated feasible resource set $R(\barposelxF) \defeq \setWithArg{\poselxR \in \resPosetR}{\tup{\barposelxF, \poselxR} \in \ldprb}$.
    By the polyhedral representation of $\ldprb$, a pair $\tup{\barposelxF, \poselxR}$ is feasible if and only if
    \begin{equation*}
        A_{\resPosetR} \poselxR \;\geq\; b - A_{\funPosetF} \barposelxF, \quad \poselxR \in \resPosetR \subseteq \mathbb{R}_{\geq 0}^{n_{\resPosetR}}.
    \end{equation*}
    Hence $R(\barposelxF)$ coincides with the feasible set of the vector optimization problem
    \[
        \mathrm{Min}_{\preceq} \; \poselxR \quad \text{s.t.} \; A_{\resPosetR} \poselxR \geq b - A_{\funPosetF} \barposelxF,\quad \poselxR \geq 0,
    \]
    where $\preceq$ is the componentwise order on $\mathbb{R}^{n_{\resPosetR}}$.
    By \cref{def:queries-dp}, the query answer at $\barposelxF$ is precisely the set of $\preceq$–minimal elements of $R(\barposelxF)$.
    These are exactly the efficient (Pareto–optimal) solutions of the above \gls{abk:molp} with objective matrix $C = I$, establishing~(i).

    \textbf{(ii)} The statement for the ``fix resources, maximize functionalities'' query is proved by an order–duality argument entirely analogous to (i).
    For fixed $\barposelxR$ the feasible functionality set is the polyhedron
    \(
    F(\barposelxR) = \setWithArg{\poselxF \in \funPosetF}{A_{\funPosetF} \poselxF \geq b - A_{\resPosetR} \barposelxR,\ \poselxF \ge 0}.
    \)
    Taking $\preceq$–maximal elements of $F(\barposelxR)$ is equivalent to taking $\preceq$–minimal elements of $-F(\barposelxR)$, so exactly as in (i) one obtains that the query result coincides with the efficient set of the \gls{abk:molp} in~(ii).

    \textbf{(iii)} Let a \gls{abk:molp} of the form
    \begin{equation*}
        \mathrm{Min}_{\preceq}\, x \quad \text{subject to} \quad A x \geq b, \quad x \geq 0
    \end{equation*}
    be given, where $x \in \mathbb{R}_{\geq 0}^n$ and $\preceq$ is the componentwise order.
    Define an \gls{abk:ldp} with functionality space $\funPosetF \defeq \{0\} \subseteq \mathbb{R}_{\geq 0}$, resource space $\resPosetR \defeq \mathbb{R}_{\geq 0}^n$, and feasible set
    \begin{equation*}
        \ldprb \defeq \setWithArg{\tup{0, x} \in \funPosetF \times \resPosetR}{A x \geq b}.
    \end{equation*}
    This is an \gls{abk:ldp}, since $\ldprb$ is a polyhedron in $\mathbb{R}^1 \times \mathbb{R}^n$ and admits the representation
    \begin{equation*}
        \ldprb = \setWithArg{\tup{\poselxF, \poselxR}}{A_{\funPosetF} \poselxF + A_{\resPosetR} \poselxR \geq b},
    \end{equation*}
    with $A_{\funPosetF} = 0$ and $A_{\resPosetR} = A$.
    The fix-\FI{functionalities}-minimize-\RI{resources} query evaluated at the unique functionality $\barposelxF = 0$ asks for the set of $\preceq$–minimal $\poselxR$ such that $\tup{0,\poselxR} \in \ldprb$, that is,~$\setWithArg{\poselxR \in \mathbb{R}_{\geq 0}^n}{A \poselxR \geq b}$ together with its $\preceq$–minimal elements.
    By definition, these minimal elements are exactly the efficient solutions of the original \gls{abk:molp}, which proves~(iii).
\end{proof}
We are now ready to define \glspl{abk:lcdp}.
\begin{definition}[\Gls{abk:lcdp}]\label{def:lcdp}
Let $\tup{\funPosetF,\resPosetR,\langle \mathcal{V},\mathcal{E}\rangle}$ be a \gls{abk:cdp} in the sense of~\cref{def:cdp}.
We call~$\tup{\funPosetF,\resPosetR,\tup{ \mathcal{V},\mathcal{E}}}$ a \emph{linear co-design problem (LCDP)} if the following conditions hold:
    \begin{enumerate}[label=(\roman*), leftmargin=*]
        \item (\emph{Linear component design problems})
              For each node $d \in \mathcal{V}$, the associated DP~$d = \langle \funPosetF_d,\resPosetR_d,\ldprb_d\rangle$
              has functionality and resource posets~$\funPosetF_d \subseteq \langle\mathbb{R}_{\ge 0}^{\,n_{\funPosetF_d}},\preceq\rangle$,~$\resPosetR_d \subseteq \langle\mathbb{R}_{\ge 0}^{\,n_{\resPosetR_d}},\preceq\rangle,$ and its feasible set $\ldprb_d \subseteq \funPosetF_d \times \resPosetR_d$ is a polyhedron (there exist~$A_{\funPosetF_d} \in \mathbb{R}^{m_d \times n_{\funPosetF_d}}$,~$A_{\resPosetR_d} \in \mathbb{R}^{m_d \times n_{\resPosetR_d}}$,~$b_d \in \mathbb{R}^{m_d}$ satisfying \cref{eq:ldp-representation}).
        \item (\emph{Linear interconnection})
              Each edge~$e \in \mathcal{E}$ is a pair $e = \tup{\tup{d_1, i_1}, \tup{d_2, j_2}}$ that links the $i_1$-th resource component of node $d_1 \in \mathcal{V}$ to the $j_2$-th functionality component of node $d_2 \in \mathcal{V}$, with the two components identified by inequality of the corresponding coordinate projections, i.e.\ $\pi_{i_1}(\resPosetR_{d_1}) \leq \pi_{j_2}(\funPosetF_{d_2})$.
    \end{enumerate}
    In other words, an \gls{abk:lcdp} is a \gls{abk:cdp} whose component design problems are all \glspl{abk:ldp} and interconnections are element-wise inequalities.
    Consequently, the system-level functionality/resource are Euclidean \glspl{abk:poset}.
\end{definition}

The following lemma shows \glspl{abk:ldp} are closed under the compositions in \cref{def:dp-interconnections}.

\begin{lemma}[Closure of \glspl{abk:ldp} under composition]\label{lem:ldp-composition}
    Let $\posetP,\posetQ,\posetR,\posetP',\posetQ'$ be Euclidean sub-posets (\cref{def:euclidean-subposet}).
    Then the following hold:
    \begin{enumerate}[label=(\alph*), leftmargin=*]
        \item \emph{Series.}
        For any $\dprba \in \ldpOf{\posetP}{\posetQ}$ and $\dprbb \in \ldpOf{\posetQ}{\posetR}$, their series composition $\dprba \mthen \dprbb$ satisfies $\dprba \mthen \dprbb \in \ldpOf{\posetP}{\posetR}$.
        \item \emph{Parallel.}
        For any $\dprba \in \ldpOf{\posetP}{\posetQ}$ and $\dprba' \in \ldpOf{\posetP'}{\posetQ'}$, their parallel composition $\dprba \stack \dprba'$ satisfies $\dprba \stack \dprba' \in \ldpOf{\posetP \posetproduct \posetP'}{\posetQ \posetproduct \posetQ'}$.
        \item \emph{Intersection.}
        For any $\dprba, \dprbb \in \ldpOf{\posetP}{\posetQ}$, their intersection $\dprba \wedge \dprbb$ satisfies $\dprba \wedge \dprbb \in \ldpOf{\posetP}{\posetQ}$.
        \item \emph{Feedback.}
        For any $\dprb \in \ldpOf{\posetP \posetproduct \posetR}{\posetQ \posetproduct \posetR}$, its trace $\traceOf{\dprb}$ satisfies $\traceOf{\dprb} \in \ldpOf{\posetP}{\posetQ}$.
    \end{enumerate}
\end{lemma}

\begin{proof}
By \cref{def:ldp}, for Euclidean sub-posets~$\funPosetF,\resPosetR$ a \gls{abk:dp}~$\dprb \in \dpOf{\funPosetF}{\resPosetR}$ is an \gls{abk:ldp} if and only if its feasible set is a polyhedron in the underlying Euclidean space.
For each composition we give the feasible set of the composite \gls{abk:dp}, show that it is a polyhedron, and conclude that the composite is an \gls{abk:ldp}.
Throughout, we use that polyhedra are closed under Cartesian products and linear images.

\textbf{Series.}
Let $\dprba \in \ldpOf{\posetP}{\posetQ}$ and $\dprbb \in \ldpOf{\posetQ}{\posetR}$ be \glspl{abk:ldp} with feasible sets
    \begin{equation*}
    \begin{aligned}
        \ldprb_a
        &= \setWithArg{\tup{\poselxP,\poselxQ}}{A_{\posetP} \poselxP + A_{\posetQ} \poselxQ \ge b_a},\\
        \ldprb_b
        &= \setWithArg{\tup{\poselxQ,\poselxRnoRes}}{B_{\posetQ} \poselxQ + B_{\posetR} \poselxRnoRes \ge b_b},
    \end{aligned}
    \end{equation*}
    for suitable matrices $\tup{A_{\posetP},A_{\posetQ},b_a}$ and $\tup{B_{\posetQ},B_{\posetR},b_b}$.
    By \cref{def:dp-interconnections}, the series connection $\dprba \mthen \dprbb \in \dpOf{\posetP}{\posetR}$ has feasible set
    \begin{equation}\label{eq:series-feasible}
        \begin{aligned}
            \ldprb_{\mathrm{series}}
            = \setWithArg{\tup{\poselxP,\poselxRnoRes}}{\exists\,\poselxQ:\;
             & A_{\posetP} \poselxP + A_{\posetQ} \poselxQ \ge b_a,\;      \\
             & B_{\posetQ} \poselxQ + B_{\posetR} \poselxRnoRes \ge b_b }.
        \end{aligned}
    \end{equation}
    In the extended variable space
    \(
        x \defeq [\poselxP^\top,\poselxQ^\top,\poselxRnoRes^\top]^\top,
    \)
    the constraints in \eqref{eq:series-feasible} can be written as
    \[
        \begin{bmatrix}
            A_{\posetP} & A_{\posetQ} & 0 \\
            0           & B_{\posetQ} & B_{\posetR}
        \end{bmatrix}
        x
        \;\ge\;
        \begin{bmatrix}
            b_a \\[1mm] b_b
        \end{bmatrix},
    \]
    which defines a polyhedron in $\tup{\poselxP,\poselxQ,\poselxRnoRes}$.
    The feasible set $\ldprb_{\mathrm{series}}$ is its projection onto the $\tup{\poselxP,\poselxRnoRes}$ coordinates, hence is again a polyhedron.
    Since the external functionality and resource posets are still $\posetP$ and $\posetR$, which are Euclidean sub-posets by assumption, $\dprba \mthen \dprbb$ is an \gls{abk:ldp}.

\textbf{Parallel.}
    Let $\dprba \in \ldpOf{\posetP}{\posetQ}$ and $\dprba' \in \ldpOf{\posetP'}{\posetQ'}$ be \glspl{abk:ldp} with feasible sets
    \begin{equation*}
    \begin{aligned}
        \ldprb_a
        &= \setWithArg{\tup{\poselxP,\poselxQ}}{A_{\posetP} \poselxP + A_{\posetQ} \poselxQ \ge b_a},\\
        \ldprb_a'
        &= \setWithArg{\tup{\poselxPprime,\poselxQprime}}{A_{\posetPprime} \poselxPprime + A_{\posetQprime} \poselxQprime \ge b_a'}.
    \end{aligned}
    \end{equation*}
    By \cref{def:dp-interconnections}, their parallel connection $\dprba \stack \dprba' \in \dpOf{\posetP \posetproduct \posetPprime}{\posetQ \posetproduct \posetQprime}$ has feasible set
    \begin{equation}\label{eq:parallel-feasible}
    \begin{aligned}
        \ldprb_{\mathrm{par}}
        &=
        \Bigl\{\tup{\tup{\poselxP,\poselxPprime}, \tup{\poselxQ,\poselxQprime}}
        \,\Big|\,\\
        &\tup{\poselxP,\poselxQ} \in \ldprb_a,\;
        \tup{\poselxPprime,\poselxQprime} \in \ldprb_a'\Bigr\}.
        \end{aligned}
    \end{equation}
    The right-hand side is the Cartesian product $\ldprb_a \times \ldprb_a'$ of two polyhedra in disjoint variables, hence a polyhedron.
    Moreover, $\posetP \times \posetP'$ and $\posetQ \times \posetQ'$ are Euclidean sub-posets (products of subsets of Euclidean spaces with componentwise order), so $\dprba \stack \dprba'$ is an \gls{abk:ldp}.

\textbf{Intersection.}
Let $\dprba, \dprbb \in \ldpOf{\posetP}{\posetQ}$ be \glspl{abk:ldp} with feasible sets
    \begin{equation*}
    \begin{aligned}
        \ldprb_a
        &= \setWithArg{\tup{\poselxP,\poselxQ}}{A_{\posetP} \poselxP + A_{\posetQ} \poselxQ \ge b_a},\\
        \ldprb_b
        &= \setWithArg{\tup{\poselxP,\poselxQ}}{B_{\posetP} \poselxP + B_{\posetQ} \poselxQ \ge b_b},
    \end{aligned}
    \end{equation*}
    for suitable $\tup{A_{\posetP},A_{\posetQ},b_a}$ and $\tup{B_{\posetP},B_{\posetQ},b_b}$. By \cref{def:dp-interconnections}, the intersection composition $\dprba \wedge \dprbb \in \dpOf{\posetP}{\posetQ}$ has feasible set
    \begin{equation}\label{eq:intersection-feasible}
    \begin{aligned}
        \ldprb_{\mathrm{int}}
        = \setWithArg{\tup{\poselxP,\poselxQ}}{
            &A_{\posetP} \poselxP + A_{\posetQ} \poselxQ \ge b_a, \\
            &B_{\posetP} \poselxP + B_{\posetQ} \poselxQ \ge b_b
        }.
    \end{aligned}
    \end{equation}
    The constraints in \cref{eq:intersection-feasible} can be written as a single system of linear inequalities by stacking the two descriptions:
    \[
        \begin{bmatrix}
            A_{\posetP} & A_{\posetQ} \\
            B_{\posetP} & B_{\posetQ}
        \end{bmatrix}
        \begin{bmatrix}
            \poselxP \\[1mm] \poselxQ
        \end{bmatrix}
        \;\ge\;
        \begin{bmatrix}
            b_a \\[1mm] b_b
        \end{bmatrix},
    \]
    which defines a polyhedron directly in $\tup{\poselxP,\poselxQ}$. 
    Since $\posetP$ and $\posetQ$ are Euclidean sub-posets, $\dprba \wedge \dprbb$ is an \gls{abk:ldp}.
    
\textbf{Feedback.}
Let $\dprb \in \ldpOf{\posetP \posetproduct \posetR}{\posetQ \posetproduct \posetR}$ be an \gls{abk:ldp} with feasible set
    \begin{equation*}
        \begin{aligned}            
        \ldprb
        =&
        \Bigl\{
            \tup{\tup{\poselxP,\poselxRnoRes^{\mathrm{in}}},\tup{\poselxQ,\poselxRnoRes^{\mathrm{out}}}}
            \,\Bigm|\,\\
            &A_{\posetP} \poselxP
            + A_{\posetQ} \poselxQ
            + A_{\posetR}^{\mathrm{in}} \poselxRnoRes^{\mathrm{in}}
            + A_{\posetR}^{\mathrm{out}} \poselxRnoRes^{\mathrm{out}}
            \ge b
        \Bigr\}.
        \end{aligned}
    \end{equation*}
    By \cref{def:dp-interconnections},~$\traceOf{\dprb} \in \dpOf{\posetP}{\posetQ}$ has feasible set
    \begin{equation}\label{eq:feedback-feasible}
        \ldprb_{\mathrm{fb}}
        =
        \setWithArg{\tup{\poselxP,\poselxQ}}{\exists\,\poselxRnoRes:\;
            \tup{\tup{\poselxP,\poselxRnoRes),(\poselxQ,\poselxRnoRes}} \in \ldprb }.
    \end{equation}
    Equivalently, $\tup{\poselxP,\poselxQ} \in \ldprb_{\mathrm{fb}}$ if and only if there exists $\poselxRnoRes$ such that~$\bigl(A_{\posetR}^{\mathrm{in}} + A_{\posetR}^{\mathrm{out}}\bigr)\poselxRnoRes
        + A_{\posetP}\poselxP
        + A_{\posetQ}\poselxQ
        \;\ge\; b.$
    The set of triples $\tup{\poselxP,\poselxQ,\poselxRnoRes}$ satisfying this inequality is a polyhedron in the extended space; $\ldprb_{\mathrm{fb}}$ is its projection onto $\tup{\poselxP,\poselxQ}$, hence a polyhedron.
    Since $\posetP$ and $\posetQ$ remain Euclidean sub-posets, $\traceOf{\dprb}$ is an \gls{abk:ldp}.
\end{proof}

We now show that every \gls{abk:lcdp} induces an \gls{abk:ldp} at the system level, so that its feasible set is a polyhedron and queries can be cast as \gls{abk:molp}.

\begin{theorem}[An LCDP induces an LDP]\label{thm:lcdp-is-ldp}
    Let $\tup{\funPosetF,\resPosetR,\langle \mathcal{V},\mathcal{E}\rangle}$ be an \gls{abk:lcdp} in the sense of \cref{def:lcdp}.
    The \gls{abk:cdp} obtained by composing the component \glspl{abk:dp} along the multigraph $\langle \mathcal{V},\mathcal{E}\rangle$ is an \gls{abk:ldp} in $\dpOf{\funPosetF}{\resPosetR}$.
    In particular, its feasible set $\ldprb_{\mathrm{tot}} \subseteq \funPosetF \times \resPosetR$ is a polyhedron.
\end{theorem}

\begin{proof}
    A \gls{abk:ldp} is constructed from a \gls{abk:lcdp} by applying compositions in \cref{def:dp-interconnections} to the components according to $\tup{\mathcal{V}, \mathcal{E}}$,
    adding splitter/merger \glspl{abk:dp} of the form
    \begin{equation*}
    \begin{aligned}
        \dprb_{\posetP, k}^s &\defeq \setWithArg{\tup{\poselxP, \tup{\posx'_{\posetP,1}, \cdots, \posx'_{\posetP,k}}}}{\poselxP \posetleq \posx'_{\posetP,i}, \, \forall i},\\
    \dprb_{\posetP, k}^m &\defeq \setWithArg{\tup{\tup{\posx_{\posetP,1}, \cdots, \posx_{\posetP,k}}, \poselxP'}}{\posx_{\posetP,i} \posetleq \poselxP', \, \forall i}
    \end{aligned}
    \end{equation*}
    when necessary \cite{zardini2023co}.
    Since composing \glspl{abk:ldp} gives \glspl{abk:ldp}, we only need to prove these splitters/mergers \glspl{abk:dp} are \glspl{abk:ldp}, which follows by their definitions.
\end{proof}

\section{Feasible-Set Construction and Query Solving for LCDPs}\label{sec:lcdp-computation}

\Cref{thm:lcdp-is-ldp} implies that any \gls{abk:lcdp} induces a system-level \gls{abk:ldp}, and therefore inherits the \gls{abk:molp}-based query semantics of \cref{lem:ldp-queries-molp}.
Computationally, this gives rise to two tasks:
(i) \emph{query evaluation}, i.e., computing Pareto-minimal resources (or Pareto-maximal functionalities) for fixed external ports, and
(ii) \emph{explicit feasible-set construction}, i.e., producing an $\mathbf{H}$-representation of the induced system-level polyhedron $\ldprb_{\mathrm{tot}}\subseteq \funPosetF\times\resPosetR$ for reuse across queries or further analysis.
Both tasks reduce to polyhedral calculus.
An \gls{abk:lcdp} can be represented as the Cartesian product of the component polyhedra intersected with a sparse linear subspace enforcing the wiring equalities; the system-level feasible set is then obtained by projection onto the external ports.
This yields two complementary computational strategies.
The \emph{monolithic} strategy (\cref{subsec:monolithic}) assembles the lifted block-angular polyhedron once and answers each query as a single \gls{abk:molp} in the full variable space, avoiding explicit projection.
The \emph{compositional} strategy (\cref{subsec:compositional}) follows the co-design graph and incrementally eliminates internal wiring variables, producing a lower-dimensional $\mathbf{H}$-representation in external coordinates at the cost of possible inequality growth.
\Cref{prop:lcdp-global-polyhedron,prop:lcdp-query-lift} formalize the global polyhedral encoding and the equivalence between projected and lifted query solving.

\begin{lemma}[Global polyhedral encoding of an \gls{abk:lcdp}]\label[lemma]{prop:lcdp-global-polyhedron}
Consider an \gls{abk:lcdp}~$\tup{\funPosetF,\resPosetR,\tup{\mathcal{V},\mathcal{E}}}$ in the sense of \cref{def:lcdp}.
For each node $d\in\mathcal{V}$, let
\(
\ldprb_d = \{z_d \in \mathbb{R}^{n_d}\mid M_d z_d \ge m_d\}
\)
be an $\mathbf{H}$-representation of its component feasible set (including the port-domain constraints $z_d \in \funPosetF_d\times \resPosetR_d$).
Let $z_{\mathrm{global}} = [z_1^\top,\dots,z_N^\top]^\top$ and define
\begin{equation*}
\begin{aligned}
M_{\mathrm{blk}} &\defeq \operatorname{block-diag}(M_1,\dots,M_N),\\
m_{\mathrm{blk}} &\defeq [m_1^\top,\dots,m_N^\top]^\top.
\end{aligned}
\end{equation*}
Let~$E z_{\mathrm{global}}=0$ be the collection of all wiring equalities induced by the edges~$\mathcal{E}$ (each row of~$E$ has exactly one~$+1$ and one~$-1$).
Then the \emph{full-space} system-level feasible set of the \gls{abk:lcdp} is the polyhedron
\begin{equation}\label{eq:ldp-fullspace}
\ldprb_{\mathrm{tot}}^{\mathrm{full}}
=
\Bigl\{
z_{\mathrm{global}}
\Big|
M_{\mathrm{blk}} z_{\mathrm{global}} \ge m_{\mathrm{blk}},
E z_{\mathrm{global}} = 0
\Bigr\}.
\end{equation}

\begin{remark}\label{rmk}
    Writing inequalities in \cref{fig:dp_def} as equalities in \cref{prop:lcdp-global-polyhedron} produces the same system-level feasible set, since for any feasible $z_\mathrm{global}$ where one constraint $x_{\R{\posetP}} < x_{\F{\posetP}}$ is strict, assigning, for instance, $\hat{x}_{\R{\posetP}} \defeq x_{\F{\posetP}}$, results in another feasible global variable because every component is monotone.
    It also provides and requires the same functionality and resource at the system level.
\end{remark}

If~$S_{\mathrm{ext}}$ is a selection matrix extracting the external ports~$z_{\mathrm{ext}} = [\poselxF^\top,\poselxR^\top]^\top$ from $z_{\mathrm{global}}$,
then the induced \emph{external-port} system-level feasible set is the projection
\begin{equation}\label{eq:ldp-projection}
\ldprb_{\mathrm{tot}}
=
\Bigl\{
z_{\mathrm{ext}}
\ \Big|\ 
\exists z_{\mathrm{global}} \in \ldprb_{\mathrm{tot}}^{\mathrm{full}}
\text{ s.t. } z_{\mathrm{ext}} = S_{\mathrm{ext}}z_{\mathrm{global}}
\Bigr\}.
\end{equation}
\end{lemma}

\begin{proof}
By definition of an \gls{abk:lcdp} (\cref{def:lcdp}), a system design is feasible if and only if each component design $z_d$ satisfies its local feasibility constraints and all interconnections identify the corresponding connected port coordinates by equality.
Stacking the local $\mathbf{H}$-representations yields the block-diagonal system~$M_{\mathrm{blk}} z_{\mathrm{global}} \ge m_{\mathrm{blk}}$, while all port identifications yield the linear equalities~$E z_{\mathrm{global}}=0$, giving~\eqref{eq:ldp-fullspace}.
Finally, the system-level feasible set retains only the external ports, i.e., it is the image of the full-space set under the coordinate projection encoded by~$S_{\mathrm{ext}}$, which is~\eqref{eq:ldp-projection}.
\end{proof}

\begin{lemma}[Lifted MOLP queries equal projected DP queries]\label[lemma]{prop:lcdp-query-lift}
Let $\ldprb_{\mathrm{tot}}$ and $\ldprb_{\mathrm{tot}}^{\mathrm{full}}$ be as in \cref{prop:lcdp-global-polyhedron}.
Let $S_{\F{F}}$ and $S_{\R{R}}$ be selection matrices such that $\poselxF=S_{\F{F}} z_{\mathrm{global}}$ and $\poselxR=S_{\R{R}} z_{\mathrm{global}}$.
Fix an external functionality $\barposelxF\in\funPosetF$ and consider the two feasible resource sets
\begin{equation*}
\begin{aligned}
\mathcal{R}_{\mathrm{ext}}(\barposelxF)
&\defeq
\{\poselxR \mid \tup{\barposelxF,\poselxR}\in \ldprb_{\mathrm{tot}}\},\\
\mathcal{R}_{\mathrm{lift}}(\barposelxF)
&\defeq
\{S_{\R{R}} z_{\mathrm{global}} \mid z_{\mathrm{global}}\in \ldprb_{\mathrm{tot}}^{\mathrm{full}},\ S_{\F{F}} z_{\mathrm{global}}=\barposelxF\}.
\end{aligned}
\end{equation*}
Then $\mathcal{R}_{\mathrm{ext}}(\barposelxF)=\mathcal{R}_{\mathrm{lift}}(\barposelxF)$, and hence the Pareto-minimal (efficient) resources returned by the DP query coincide with the efficient outcomes of the lifted MOLP
\begin{equation}\label{eq:monolithic-molp-query}
\begin{aligned}
&\mathrm{Min}_{\preceq} \ S_{\R{R}} z_{\mathrm{global}}\\
&\text{s.t.} \quad
M_{\mathrm{blk}}z_{\mathrm{global}} \ge m_{\mathrm{blk}},\ 
E z_{\mathrm{global}}=0,\ 
S_{\F{F}} z_{\mathrm{global}} = \barposelxF.
\end{aligned}
\end{equation}
An analogous statement holds for the ``fix \R{resources}, maximize \F{functionalities}'' query.
\end{lemma}

\begin{proof}
By \cref{prop:lcdp-global-polyhedron}, $\tup{\barposelxF,\poselxR}\in \ldprb_{\mathrm{tot}}$ holds if and only if there exists a full assignment $z_{\mathrm{global}}$ such that $z_{\mathrm{global}}\in \ldprb_{\mathrm{tot}}^{\mathrm{full}}$, $S_{\F{F}} z_{\mathrm{global}}=\barposelxF$, and $S_{\R{R}} z_{\mathrm{global}}=\poselxR$.
This is exactly the equality $\mathcal{R}_{\mathrm{ext}}(\barposelxF)=\mathcal{R}_{\mathrm{lift}}(\barposelxF)$.
Since both optimization problems compare resource vectors under the same componentwise order $\preceq$ and have identical attainable resource sets, their efficient outcomes coincide.
The second query follows by the same argument with $(S_{\F{F}},S_{\R{R}})$ exchanged and the order reversal accounted for as in \cref{lem:ldp-queries-molp}.
\end{proof}

\subsection{Monolithic formulation}\label{subsec:monolithic}
The \emph{monolithic} strategy implements the lifted encoding of \cref{prop:lcdp-global-polyhedron} by introducing \emph{all} component port variables as decision variables, enforcing (i) each component's local feasibility constraints and (ii) all wiring equalities.
This yields the full-space polyhedron $\ldprb_{\mathrm{tot}}^{\mathrm{full}}$ in~\eqref{eq:ldp-fullspace}, whose constraint matrix is \emph{block-angular}: a block-diagonal part $M_{\mathrm{blk}}$ (one block per component) plus a very sparse coupling part $E$ (each wiring row has exactly two nonzeros).
For query evaluation, {no explicit projection} onto external ports is required: by \cref{prop:lcdp-query-lift}, the projected \gls{abk:dp} query is equivalent to solving the lifted \gls{abk:molp}~\eqref{eq:monolithic-molp-query} (and its dual) in the full variable space.
This is typically the most efficient route for \emph{single-shot} queries, because it avoids polyhedral projection and lets the solver handle internal feasibility implicitly.
\cref{alg:monolithic} details the construction of the lifted constraint system in practice.

\begin{remark}[Monolithic formulation and Dantzig--Wolfe decomposition]
The lifted feasibility system in~\eqref{eq:ldp-fullspace} is block-angular, the classical structural setting for \emph{Dantzig--Wolfe} (DW) decomposition.
For scalar-objective instances (e.g., minimize a single weighted sum of resources), DW can exploit the separation between local component constraints and global couplings to solve large instances via a master problem plus one subproblem per node~\cite{vanderbeck2006generic}.
For \gls{abk:molp} queries, DW is not directly applicable as a complete Pareto-front method, but the same sparsity remains valuable: many \gls{abk:molp} algorithms (including Benson-type outer approximation) reduce computation to solving sequences of LPs, each inheriting the same coupling structure.
This suggests that DW-style decomposition can still be used \emph{inside} Pareto-front algorithms to accelerate the repeated LP solves.
\end{remark}

\begin{algorithm}[tb]
    \caption{Monolithic feasible set for composite LCDP}\label{alg:monolithic}
    \begin{algorithmic}[1]
        \Require \gls{abk:lcdp} $\langle\funPosetF,\resPosetR,\langle \mathcal{V},\mathcal{E}\rangle\rangle$; for each $d \in \mathcal{V}$ ($\left|\mathcal{V}\right| = N$), component \gls{abk:ldp} with $\ldprb_d = \{ \tup{x_{\funPosetF,d},x_{\resPosetR,d}} : A_{\funPosetF,d}x_{\funPosetF,d} + A_{\resPosetR,d}x_{\resPosetR,d} \ge b_d \}$
        \Ensure System-level feasible set $\ldprb_{\mathrm{tot}}^{\mathrm{full}}$ in $\mathbf{H}$-representation $M_{\mathrm{blk}} z_{\mathrm{global}} \ge m_{\mathrm{blk}}$ (full-space variables)
        \State \textbf{Local formulation:} For each node $d \in \mathcal{V}$, set $z_d \gets [x_{\funPosetF,d}^\top,\, x_{\resPosetR,d}^\top]^\top$, $M_d \gets [A_{\funPosetF,d}\;\; A_{\resPosetR,d}]$, $m_d \gets b_d$
        \State \textbf{Aggregation:} Form $z_{\mathrm{global}} \gets [z_1^\top,\, \ldots,\, z_N^\top]^\top$, $M_{\mathrm{blk}} \gets \operatorname{block-diag}(M_1,\ldots,M_N)$, $m_{\mathrm{blk}} \gets [m_1^\top,\, \ldots,\, m_N^\top]^\top$
        \State System $M_{\mathrm{blk}} z_{\mathrm{global}} \ge m_{\mathrm{blk}}$ is block-diagonal (one block per node)
        \State \textbf{Edge constraints:} For each edge $e \in \mathcal{E}$ (series or feedback, \cref{def:lcdp}), the edge links one functionality and one resource. In $z_{\mathrm{global}}$ (with $z_d = [x_{\funPosetF,d}^\top,\, x_{\resPosetR,d}^\top]^\top$), find the indices $i_{\mathrm{f}}$, $i_{\mathrm{r}}$ of these two entries. Append to $M_{\mathrm{blk}}$ one row with $+1$ at $i_{\mathrm{f}}$, $-1$ at $i_{\mathrm{r}}$, zeros elsewhere, and append $0$ to $m_{\mathrm{blk}}$ (so $z_{i_{\mathrm{f}}} - z_{i_{\mathrm{r}}} \ge 0$; add a second row $-z_{i_{\mathrm{f}}} + z_{i_{\mathrm{r}}} \ge 0$ to enforce equality). Thus each edge adds one (or two) rows to the constraint matrix.
        \State \Return $\ldprb_{\mathrm{tot}}^{\mathrm{full}} = \{ z_{\mathrm{global}} : M_{\mathrm{blk}} z_{\mathrm{global}} \ge m_{\mathrm{blk}} \}$
    \end{algorithmic}
\end{algorithm}

\subsection{Compositional approach and variable elimination}\label{subsec:compositional}
The monolithic polyhedron~\eqref{eq:ldp-fullspace} is built in a lifted space containing \emph{all} component port variables, including many internal wiring coordinates that are not of direct system-level interest.
When one needs an explicit system-level feasible set $\ldprb_{\mathrm{tot}}\subseteq \funPosetF\times\resPosetR$ (e.g., to amortize repeated queries or to export a reduced model), it can be advantageous to eliminate internal variables and work directly in external coordinates.
The \emph{compositional} approach does exactly this by traversing the co-design graph and applying the same polyhedral operations that appear in the closure proof (\cref{lem:ldp-composition}): Cartesian products (parallel), constraint stacking (intersection), linear coupling constraints (series/feedback), and projections (variable elimination).

\subsubsection{Composition as polyhedral calculus}
At the level of feasible sets, the four \gls{abk:dp} interconnections can be written uniformly as: (i) take a Cartesian product of the participating polyhedra, (ii) intersect with a linear subspace enforcing port identifications, and (iii) project out the internal coordinates.
For example, suppose an edge identifies one resource coordinate of subsystem $u$ with one functionality coordinate of subsystem $v$; write this equality as $a^\top [z_u^\top,z_v^\top]^\top = 0$ with a sparse row $a^\top$ containing one $+1$ and one $-1$.
Then the \emph{series contraction} of $u$ and $v$ is
\begin{equation}\label{eq:series-polyhedral-calculus}
P_{uv}
=
\Pi_{\mathrm{keep}}
\Bigl(
(P_u \times P_v)\ \cap\ \{\tup{z_u,z_v}\mid a^\top[z_u^\top,z_v^\top]^\top = 0\}
\Bigr),
\end{equation}
where $\Pi_{\mathrm{keep}}$ removes the internal connected coordinate (and possibly other internal ports that cease to be boundary variables after contraction).
Parallel composition is simply $P_u\times P_v$ in disjoint coordinates.
Intersection of two components $P_a,P_b$ sharing the same external ports is even simpler: the composed feasible set is $P_a\cap P_b$, realized by vertically stacking the two inequality systems without introducing any new variables or projections.
Feedback (trace) is the special case where both coordinates live in the same polyhedron $P$: intersect with an equality subspace identifying the two ports, then project out the identified internal variable.

\subsubsection{Graph-guided elimination schedule}
The goal is to choose an elimination order that keeps intermediate polyhedra manageable.
We use the multigraph $\langle\mathcal{V},\mathcal{E}\rangle$ to guide a \emph{projection schedule}:
\begin{itemize}[leftmargin=*]
    \item \textbf{Intersection preprocessing:} Whenever two or more components share the \emph{same} external ports (parallel edges in the multigraph), their inequality systems are stacked to form a single merged node before any series or feedback step, a projection-free operation that just concatenates constraints.
    \item \textbf{Parallel decomposition:} Connected components of the underlying undirected graph are independent and therefore compose in parallel.
    \item \textbf{Series-first skeleton:} Within each connected component, choose a spanning forest (e.g., via DFS). The forest edges define a \emph{series contraction skeleton} that reduces the component to a single composite node.
    \item \textbf{Feedback-late edges:} The remaining edges (those not in the spanning forest) are precisely the cycle-closing constraints. After the series skeleton is contracted, each such edge becomes an internal wiring equality within the composite node and is handled by a feedback (trace) elimination.
\end{itemize}
This is algorithmically convenient: intersection merges reduce the node count without any projection, series contractions primarily reduce the graph size, while feedback eliminations remove the remaining internal degrees of freedom.

\subsubsection{Projection primitives and practical considerations}
Each contraction step requires polyhedral projection. 
In low dimensions, \gls{abk:fme} is simple and effective; however, it can generate up to $p\cdot q$ new inequalities when eliminating one variable, where $p$ and $q$ are the numbers of inequalities with positive/negative coefficients in that variable.
Therefore, practical implementations typically combine projection with:
(i) \emph{redundancy removal} (LP-based checks to delete implied inequalities),
(ii) \emph{variable scaling/normalization} to reduce numerical sensitivity, and
(iii) \emph{hybrid cutoffs} that switch to the monolithic representation once the inequality count exceeds a budget.

\begin{lemma}[Correctness of compositional elimination]\label[lemma]{prop:compositional-correct}
\cref{alg:compositional} returns a polyhedron in the external coordinates whose feasible set equals the system-level feasible set $\ldprb_{\mathrm{tot}}$ defined by the \gls{abk:lcdp} semantics (equivalently, the projection in~\eqref{eq:ldp-projection}).
\end{lemma}

\begin{proof}
Fix a connected component $\mathcal{C}$.
Any components sharing identical external ports are first merged by stacking their inequality systems; this realizes DP intersection (\cref{def:dp-interconnections}), since the stacked system $[M_a; M_b]\,z \ge [m_a; m_b]$ defines exactly $P_a \cap P_b$.
Each subsequent series contraction step in \cref{alg:compositional} replaces two neighboring subsystems by their series interconnection, implemented as ``product, wiring equality, and projection'' as in~\eqref{eq:series-polyhedral-calculus}; this operation is exactly the feasible-set realization of \gls{abk:dp} series composition in \cref{def:dp-interconnections}.
Thus, after contracting all edges of a spanning forest, the resulting composite node represents the feasible set of the parallel-free, cycle-free composition of all subsystems in $\mathcal{C}$.

Each remaining (non-forest) edge enforces an additional wiring equality that closes a cycle; imposing that equality and projecting out the internal identified port is precisely the feasible-set realization of DP feedback/trace (\cref{def:dp-interconnections}).
Therefore, after all cycle-closing edges are processed, the component polyhedron $P_{\mathcal{C}}$ equals the feasible set induced by composing all \glspl{abk:dp} in $\mathcal{C}$ and restricting to its external ports.

Finally, distinct connected components share no wiring constraints, so their feasible sets compose by Cartesian product, matching DP parallel composition.
Hence $\prod_k P_{\mathcal{C}_k}$ equals the system-level feasible set in external coordinates~$\ldprb_{\mathrm{tot}}$.
\end{proof}

\begin{remark}[Relation between monolithic and compositional representations]
The monolithic formulation constructs~$\ldprb_{\mathrm{tot}}^{\mathrm{full}}$ in~\eqref{eq:ldp-fullspace} and then (optionally) projects once onto the external ports.
The compositional formulation performs the same projection incrementally during graph reduction.
Both methods therefore compute the same external feasible set; they differ only in the elimination, trading off dimensionality (fewer variables) against potential inequality growth under projection.
\end{remark}

\begin{remark}[MOLP solution via Benson's algorithm]\label{rem:benson-algorithm}
    Each query on an \gls{abk:ldp} reduces to a polyhedral \gls{abk:molp} (\cref{lem:ldp-queries-molp}), or more generally a VLP (see the remark following \cref{def:upper-image-molp}).
    For polyhedral \gls{abk:molp}, the Benson outer-approximation algorithm is finite~\cite[Theorem~3.1]{benson1998outer} and recovers efficient extreme outcomes~\cite[Theorem~3.2]{benson1998outer}; further analysis shows recovery of weakly efficient outcomes as well~\cite[Theorem~3.3]{benson1998further}.
\end{remark}

\begin{algorithm}[tb]
    \caption{Compositional feasible set via graph traversal}\label{alg:compositional}
    \begin{algorithmic}[1]
        \Require \gls{abk:lcdp} $\tup{\funPosetF,\resPosetR,\langle \mathcal{V},\mathcal{E}}$; for each $d \in \mathcal{V}$, component \gls{abk:ldp} with $(M_d, m_d)$ s.t.\ $\ldprb_d = \{ z_d : M_d z_d \ge m_d \}$
        \Ensure System-level feasible set $\ldprb_{\mathrm{tot}}$ in $\mathbf{H}$-representation $M_G z_{\mathrm{ext}} \ge m_G$ (external-port variables)
        \State \textbf{Classify edges:} Run DFS on the directed graph (resource $\to$ functionality). Partition $\mathcal{E} = E_{\mathrm{ser}} \cup E_{\mathrm{fb}}$ (tree/forward = series, back edges = feedback). Find connected components of the underlying undirected graph for parallel grouping.
        \State \textbf{Merge intersections:} For every group of nodes sharing the same external ports (parallel edges in the multigraph), stack their inequality systems $[M_a; M_b]\,z \ge [m_a; m_b]$ into a single node. No projection is needed.
        \State \textbf{Process series:} Within each component, compute a topological order of nodes (DAG from removing back-edges). In this order, maintain a current composite LDP; for each series edge $\tup{u,v}$, compose with $v$ in series as in the proof of \cref{lem:ldp-composition}: form the stacked polyhedron in $\tup{x_P, x_Q, x_R}$, then apply \gls{abk:fme} to project onto $\tup{x_P, x_R}$. Replace $u$, $v$ by the composite node and update the graph.
        \State \textbf{Process parallel:} The previous step yields one composite LDP per connected component. Compose these in parallel as in \cref{lem:ldp-composition}: Cartesian product (block-diagonal $M$, stacked $m$), no shared variables.
        \State \textbf{Process feedback:} For each remaining back edge (feedback loop), apply the trace as in \cref{lem:ldp-composition}: impose equality of the linked resource and functionality, then project out that variable via \gls{abk:fme}. Update the composite.
        \State \Return $\ldprb_{\mathrm{tot}} = \{ z_{\mathrm{ext}} : M_G z_{\mathrm{ext}} \ge m_G \}$
    \end{algorithmic}
\end{algorithm}

\section{Convex feasible-set approximation}\label{subsec:convex-epsilon}

The monolithic and compositional strategies of \cref{subsec:monolithic,subsec:compositional} exploit that every component is an \gls{abk:ldp}, so that the induced system-level feasible set is a polyhedron and every query reduces to an \gls{abk:molp}.
In practice, however, subsystem models often involve nonlinear couplings that remain convex but are not polyhedral (e.g., second-order cone constraints, semidefinite constraints, or the convex bilinear inequality $p\,m\ge c$ over $p,m\ge 0$ that appears in \cref{subsec:rover}).
In such cases, the induced system-level feasible set is still a closed convex {upper} set, but it is not an \gls{abk:ldp}. This section shows how convex co-design models can still be handled by the \gls{abk:ldp} pipeline via {polyhedral outer approximation}.

\subsubsection{From convex DPs to approximate LDPs}
We say that a \gls{abk:dp} $\dprb\in\dpOf{\funPosetF}{\resPosetR}$ is \emph{convex} if its feasible set
$\dprb\subseteq \funPosetF\op\times\resPosetR\subseteq\mathbb{R}^{n_{\funPosetF}+n_{\resPosetR}}$
is nonempty, closed, and convex (in the ambient Euclidean topology).\footnote{Linear images (projections) of closed convex sets need not be closed in general.
When this matters, one can either assume regularity guaranteeing closedness of the induced projections, or work with the closure of the induced set (a convention in convex vector optimization when defining upper images).}
Convex sets admit polyhedral approximations.

\begin{definition}[Polyhedral outer approximation]\label{def:polyhedral-outer-approx}
Let $C\subseteq \mathbb{R}^n$ be nonempty.
A polyhedron $P\subseteq\mathbb{R}^n$ is a \emph{polyhedral outer approximation} of $C$ if $C\subseteq P$.
If, in addition, $C$ is an upper set in a cone-ordered space $\tup{K,\preceq_K}$ (i.e., $C+K\subseteq C$), then $P$ is an \emph{upper-set outer approximation} if $P+K\subseteq P$.
\end{definition}

An outer approximation yields an approximate \gls{abk:ldp} (hence \gls{abk:molp}-solvable queries), but its accuracy must be quantified in a way that remains meaningful for \emph{unbounded} sets.
Indeed, co-design query-feasible resource sets are inherently unbounded.
By the upper-set property of \glspl{abk:dp} (\cref{def:dp}), if $\poselxR$ is feasible for $\barposelxF$, then any $\poselxR'\succeq \poselxR$ is feasible as well.
For unbounded closed convex sets, the classical Hausdorff distance between the set and its approximation is usually infinite unless their recession cones coincide exactly;
see \cite[Proposition~3.5]{doerfler2022approximation}.
We therefore adopt the $\left(\varepsilon,\delta\right)$-approximation framework of~\cite{doerfler2022approximation}, which separates the error on the bounded part, measured by $\varepsilon$, from the asymptotic error in the recession cone, measured by $\delta$.
We recall the main notions from~\cite{doerfler2022approximation}.

\begin{definition}[Recession cone, line-freeness, self-boundedness]\label{def:recession-selfbounded}
Let $C\subseteq\mathbb{R}^n$ be nonempty, closed, and convex.
Its \emph{recession cone} is
\begin{equation*}
0^+C \defeq \{d\in\mathbb{R}^n \mid x+td\in C,\ \forall x\in C,\ \forall t\ge 0\}.
\end{equation*}
$C$ is \emph{line-free} if $0^+C\cap(-0^+C)=\{0\}$ (i.e., the recession cone is pointed).
If $0^+C\neq\{0\}$, the set $C$ is \emph{self-bounded} if there exists $y\in\mathbb{R}^n$ such that $C\subseteq y+0^+C$~\cite[Definition~3.2]{doerfler2022approximation}.
\end{definition}

\noindent For nonempty sets~$C_1,C_2\subseteq\mathbb{R}^n$, the \emph{excess} of~$C_1$ over~$C_2$ is
\begin{equation*}
    e[C_1,C_2]\defeq \sup_{c_1\in C_1}\ \inf_{c_2\in C_2}\ \|c_1-c_2\|.
\end{equation*}
The Hausdorff distance is
\begin{equation*}
    d_{\mathrm H}(C_1,C_2)\defeq \max\{e[C_1,C_2],e[C_2,C_1]\}
\end{equation*}
For nonempty closed convex cones $K_1,K_2$, the \emph{truncated Hausdorff distance} is
\begin{equation*}
    \bar d_{\mathrm H}(K_1,K_2)\defeq d_{\mathrm H}(K_1\cap B_1(0),K_2\cap B_1(0)).
\end{equation*}

\begin{definition}[$\left(\varepsilon,\delta\right)$-approximation {\cite[Definition~4.2]{doerfler2022approximation}}]\label{def:eps-delta-approx}
Let $C\subseteq\mathbb{R}^n$ be nonempty, closed, convex, and line-free.
A line-free polyhedron $P\subseteq\mathbb{R}^n$ is an \emph{$\left(\varepsilon,\delta\right)$-approximation} of $C$ if: (i)~$e[\operatorname{vert}P,\,C]\le \varepsilon$, (ii)~$\bar d_{\mathrm H}(0^+P,\,0^+C)\le \delta$, and (iii)~$C\subseteq P$.
\end{definition}

We now collect two results from~\cite{doerfler2022approximation} that we will use.

\begin{theorem}[Existence in Hausdorff distance and convergence of approximations]\label{thm:doerfler-key}
Let $C\subseteq\mathbb{R}^n$ be nonempty and closed.
\begin{enumerate}[label=(\alph*)]
\item\emph{(Hausdorff-approximability)}
If $C$ is convex, $0^+C$ is polyhedral, and $e[C,0^+C]<\infty$, then for every $\varepsilon>0$ there exists a finite set $V$ such that
$d_{\mathrm H}\!\bigl(C,\operatorname{conv}V+0^+C\bigr)\le \varepsilon$~\cite[Theorem~3.1]{doerfler2022approximation}.
In particular, the approximating polyhedra can be chosen with $0^+P=0^+C$.
\item\emph{(Convergence)}
If $C$ is additionally line-free and $P^\nu$ is an $(\varepsilon^\nu,\delta^\nu)$-approximation of $C$ with $(\varepsilon^\nu,\delta^\nu)\to(0,0)$, then $P^\nu\to C$ in the sense of Painlev\'e--Kuratowski~\cite[Theorem~4.9]{doerfler2022approximation}.
\end{enumerate}
\end{theorem}

\subsubsection{Specialization to co-design resource queries}\label{subsec:cod-query-specialization}
While the approximation in \Cref{def:eps-delta-approx} is generally described by two parameters, a useful simplification arises for co-design \emph{resource queries}.
Because feasible resource sets are upper sets in a cone-ordered resource space, their recession cone is necessarily equal to the resource cone.
Consequently, for the standard Euclidean resource poset $\resPosetR=\mathbb{R}_{\ge 0}^{n_\resPosetR}$, the recession-cone part of the approximation can be matched \emph{exactly} (i.e., $\delta=0$), and the approximation task reduces to controlling~$\varepsilon$.

Fix a \gls{abk:dp} $\dprb\in\dpOf{\funPosetF}{\resPosetR}$ over Euclidean resource coordinates, and fix a functionality $\barposelxF\in\funPosetF$.
Define the \emph{feasible resource} $\mathcal{R}(\barposelxF)$ set and the corresponding \emph{query graph} $\mathcal{G}_{\barposelxF}$:
\begin{equation*}
\begin{aligned}
\mathcal{R}(\barposelxF)\defeq\{\poselxR\in\resPosetR \mid \tup{\barposelxF,\poselxR}\in\dprb\},\\
\mathcal{G}_{\barposelxF}\defeq \{\tup{\barposelxF,\poselxR}\mid \poselxR\in\mathcal{R}(\barposelxF)\}.
\end{aligned}
\end{equation*}
By the upper-set property of \glspl{abk:dp} (\cref{def:dp}), $\mathcal{R}(\barposelxF)$ is an upper set in the resource poset.

\begin{lemma}[Query-graph recession structure]\label[lemma]{lem:query-graph-recession}
Let $\dprb\in\dpOf{\funPosetF}{\resPosetR}$ have nonempty, closed, convex feasible set in $\funPosetF\op\times\resPosetR$.
Fix $\barposelxF\in\funPosetF$ and assume $\mathcal{R}(\barposelxF)\neq\varnothing$.
Then:
\begin{enumerate}[label=(\alph*)]
\item $0^+\mathcal{G}_{\barposelxF}
= \{0\}^{n_{\funPosetF}}\times 0^+\mathcal{R}(\barposelxF)$.
\item If $\resPosetR=\resPosetR_1\times\cdots\times\resPosetR_k$ is a product of nonempty closed convex cones and $\mathcal{R}(\barposelxF)$ is an upper set in $\resPosetR$, then~$0^+\mathcal{R}(\barposelxF)=0^+\resPosetR=\resPosetR.$
\end{enumerate}
\end{lemma}

\begin{proof}
\textbf{(a)}
By definition,~$\mathcal{G}_{\barposelxF}$ fixes the functionality coordinates at $\barposelxF$.
A direction~$\tup{d_{\F{F}},d_{\R{R}}}$ belongs to $0^+\mathcal{G}_{\barposelxF}$ if and only if for all $t\ge 0$,~$\tup{\barposelxF,\poselxR}+t\tup{d_{\F{F}},d_{\R{R}}}\in \mathcal{G}_{\barposelxF}$, which forces $\barposelxF+td_{\F{F}}=\barposelxF$ for all $t\ge 0$ and hence $d_{\F{F}}=0$.
The resource direction $d_{\R{R}}$ is then feasible for all $t$ exactly when $d_{\R{R}}\in 0^+\mathcal{R}(\barposelxF)$.
This proves the identity.

\textbf{(b)}
We show both inclusions.

\emph{(i) $\,\resPosetR\subseteq 0^+\mathcal{R}(\barposelxF)$.}
Let $d\in\resPosetR$ and $\poselxR\in \mathcal{R}(\barposelxF)$.
Since $\resPosetR$ is a cone, $td\in\resPosetR$ for all $t\ge 0$ and $\poselxR+td\succeq \poselxR$.
Because $\mathcal{R}(\barposelxF)$ is an upper set in $\resPosetR$, $\poselxR+td\in \mathcal{R}(\barposelxF)$ for all $t\ge 0$.
Hence $d\in 0^+\mathcal{R}(\barposelxF)$.

\emph{(ii) $\,0^+\mathcal{R}(\barposelxF)\subseteq \resPosetR$.}
Let $d\in 0^+\mathcal{R}(\barposelxF)$ and pick any $\poselxR\in\mathcal{R}(\barposelxF)\subseteq \resPosetR$.
Then $\poselxR+td\in \mathcal{R}(\barposelxF)\subseteq \resPosetR$ for all $t\ge 0$.
Since $\resPosetR$ is a cone, for every $t>0$ we have~$(\poselxR+td)/t=d+\poselxR/t \in \resPosetR.$
Letting $t\to\infty$ and using closedness of $\resPosetR$ yields $d\in \resPosetR$.

Finally, since $\resPosetR$ is a closed convex cone, $0^+\resPosetR=\resPosetR$.
\end{proof}

\begin{lemma}[Recession cone of an upper feasible resource set]\label[lemma]{lem:upper-set-recession}
Let $K\subseteq\mathbb{R}^n$ be a nonempty closed convex cone ordered by $x\preceq_K y \Leftrightarrow y-x\in K$.
Let $C\subseteq K$ be nonempty and an upper set in $K$ (i.e., $C+K\subseteq C$).
Then $0^+C = K$.
\end{lemma}

\begin{proof}
This is exactly \cref{lem:query-graph-recession}(b) with $\resPosetR=K$ and $C=\mathcal{R}(\barposelxF)$, but we repeat the short argument for completeness.
First, $K\subseteq 0^+C$: for any $d\in K$ and $x\in C$, $x+td\in x+K\subseteq C$ for all $t\ge 0$.
For the reverse inclusion, let $d\in 0^+C$ and choose $x\in C\subseteq K$.
Then $x+td\in C\subseteq K$ for all $t\ge 0$.
Since $K$ is a cone, $d+x/t\in K$ for all $t>0$; letting $t\to\infty$ and using closedness gives $d\in K$.
\end{proof}

\begin{theorem}[Natural \texorpdfstring{$\delta=0$}{delta=0} guarantee for convex co-design resource queries]\label{thm:epsilon-linear-approx}
Assume the system resource poset is the standard Euclidean cone $\resPosetR=\mathbb{R}_{\ge 0}^{n_\resPosetR}$.
Fix $\barposelxF\in\funPosetF$ and let $C\defeq \mathcal{R}(\barposelxF)$ be the feasible resource set, assumed nonempty, closed, and convex.
Then:
\begin{enumerate}[label=(\alph*)]
\item $0^+C=\mathbb{R}_{\ge 0}^{n_\resPosetR}$ and $C$ is line-free and self-bounded.
\item For every $\varepsilon>0$ there exists a polyhedron $P\supseteq C$ with $0^+P=0^+C$ such that $P$ is an $(\varepsilon,0)$-approximation of $C$ in the sense of \cref{def:eps-delta-approx}.
\item Any sequence of $(\varepsilon^\nu,0)$-approximations with $\varepsilon^\nu\to 0$ converges to $C$ in the sense of Painlev\'e--Kuratowski.
\end{enumerate}
\end{theorem}

\begin{proof}
\textbf{(a)}
Since $C$ is an upper set in $\mathbb{R}_{\ge 0}^{n_\resPosetR}$, \cref{lem:upper-set-recession} yields $0^+C=\mathbb{R}_{\ge 0}^{n_\resPosetR}$.
The cone $\mathbb{R}_{\ge 0}^{n_\resPosetR}$ is pointed, hence $0^+C\cap(-0^+C)=\{0\}$ and $C$ is line-free.
Self-boundedness holds because $C\subseteq \mathbb{R}_{\ge 0}^{n_\resPosetR}=0+0^+C$.

\textbf{(b)}
We apply \cref{thm:doerfler-key}(a).
Here $0^+C=\mathbb{R}_{\ge 0}^{n_\resPosetR}$ is polyhedral.
Moreover, since $C\subseteq 0^+C$, we have $e[C,0^+C]=0<\infty$.
Hence for any $\varepsilon>0$ there exists a polyhedron $P=\operatorname{conv}V+0^+C$ with
$d_{\mathrm H}(C,P)\le \varepsilon$ and $0^+P=0^+C$.
Then $e[\operatorname{vert}P,C]\le e[P,C]\le d_{\mathrm H}(C,P)\le \varepsilon$, so $P$ is an $(\varepsilon,0)$-approximation.

\textbf{(c)}
This is \cref{thm:doerfler-key}(b) with $\delta^\nu=0$.
\end{proof}

\begin{lemma}[From $\varepsilon$-outer approximation to guaranteed feasibility by inflation]\label[proposition]{prop:epsilon-inflation-feasible}
Let $C\subseteq \mathbb{R}_{\ge 0}^n$ be a nonempty upper set.
\begin{enumerate}[label=(\alph*)]
\item If $P\supseteq C$ satisfies $e[P,C]\le \varepsilon$, then every $p\in P$ admits a certified-feasible inflated point~$p^{\mathrm{safe}}\defeq p+\varepsilon\,\mathbf{1}\ \in\ C.$

\item If $P\supseteq C$ satisfies only $e[\operatorname{vert}P,C]\le \varepsilon$, then the same conclusion holds for every \emph{vertex} $v\in\operatorname{vert}P$:
$v+\varepsilon\mathbf{1}\in C$.
\end{enumerate}
\end{lemma}

\begin{proof}
We prove (a); (b) is the same argument restricted to $p=v\in\operatorname{vert}P$.
Fix $p\in P$.
By $e[P,C]\le\varepsilon$ there exists $c\in C$ such that $\|p-c\|\le \varepsilon$.
Hence $|p_i-c_i|\le \varepsilon$ for each coordinate $i$, so $c\preceq p+\varepsilon\mathbf{1}$ componentwise.
Since $C$ is an upper set, $p+\varepsilon\mathbf{1}\in C$.
\end{proof}

\begin{remark}[Practical implications for \gls{abk:ldp} pipeline]\label{rem:eps-approx-implications}
\Cref{thm:epsilon-linear-approx} explains why, for standard nonnegative resource coordinates, the recession-cone part of~$\left(\varepsilon,\delta\right)$-accuracy can be made exact ($\delta=0$): the unbounded directions of every convex query-feasible resource set are fixed by monotonicity and equal $\mathbb{R}_{\ge 0}^{n_\resPosetR}$.
Thus accuracy control reduces to the bounded-part error~$\varepsilon$.

Operationally, one can build polyhedral outer approximations by adding supporting halfspaces (e.g., tangent planes from subgradients of convex constraints), yielding an approximate \gls{abk:ldp} that can be queried via \gls{abk:molp}.
The obtained resource vectors are optimistic in general, but \cref{prop:epsilon-inflation-feasible} shows how an $\varepsilon$ bound can be converted into a simple, explicit feasibility margin.
In the rover case study (\cref{subsec:rover}), the linearization parameter $N$ (number of tangent-plane cuts) controls the tightness of the outer approximation; empirically $\varepsilon$ decreases as $N$ grows, and the limiting behavior is captured by \cref{thm:epsilon-linear-approx}(c).
\end{remark}

\subsubsection{Approximation under co-design composition}

So far, we have discussed how a single convex DP can be approximated by an \gls{abk:ldp}.
We now extend this idea to co-design problems involving multiple convex design problems.
Replacing each convex component feasible set by a polyhedral outer approximation yields a polyhedral DP, i.e., an \gls{abk:ldp}.
Importantly, {outer approximations propagate through co-design composition}.

\begin{lemma}[Outer approximations are preserved by composition]\label{lem:outer-approx-composition}
Consider a finite co-design interconnection built from classic operations (\cref{def:dp-interconnections}).
Let $\{C_d\}_{d\in\mathcal V}$ be the true component feasible sets and $\{P_d\}_{d\in\mathcal V}$ polyhedral outer approximations with $C_d\subseteq P_d$ for all $d$.
Let $C_{\mathrm{tot}}$ and $P_{\mathrm{tot}}$ be the induced system-level feasible sets obtained by composing $\{C_d\}$ and $\{P_d\}$ along the same interconnection graph.
Then $C_{\mathrm{tot}}\subseteq P_{\mathrm{tot}}$.
Moreover, if each $P_d$ is an \emph{upper-set} outer approximation with respect to the corresponding resource cone, then $P_{\mathrm{tot}}$ is an upper-set outer approximation of $C_{\mathrm{tot}}$ with respect to the system resource cone.
\end{lemma}

\begin{proof}
At the feasible-set level, each of the three interconnections is a combination of:
(i) Cartesian products,
(ii) intersection with a linear subspace (wiring equalities), and
(iii) linear projection (elimination of internal coordinates); see \cref{rem:composition-set-operations}.
Each of these operations is monotone with respect to set inclusion:
if $A\subseteq B$ then $A\times D\subseteq B\times D$, $A\cap S\subseteq B\cap S$, and $\Pi(A)\subseteq \Pi(B)$.
Applying these facts along a finite composition sequence yields $C_{\mathrm{tot}}\subseteq P_{\mathrm{tot}}$.

For the upper-set statement, note that if each component approximation satisfies~$P_d+K_d\subseteq P_d$ for its cone~$K_d$, then Cartesian products preserve this property with product cones, intersections with linear subspaces (since Minkowski addition commutes with affine subspaces), and projections since~$\Pi(P+K)=\Pi(P)+\Pi(K)$ and $\Pi(K)$ is the projected cone governing the remaining resource coordinates.
\end{proof}

\begin{remark}[Outer vs.\ inner approximation and conservativeness]
An outer approximation $P\supseteq C$ is conservative for \emph{infeasibility certification} (if $r\notin P$ then certainly $r\notin C$), but it is generally \emph{optimistic} for \emph{minimization} queries (minimizing over $P$ can return resources that are not feasible for $C$).
The~$\left(\varepsilon,\delta\right)$ framework below quantifies this optimism; additionally, for upper sets one can convert an $\varepsilon$-accurate outer point into a guaranteed-feasible point by a simple inflation margin (see \cref{prop:epsilon-inflation-feasible}).
\end{remark}

\subsubsection{Recession-cone calculus for co-design interconnections}\label{subsec:recession-calculus}

The feasible-set realizations of parallel/series/feedback interconnections are built from Cartesian products, affine intersections (wiring equalities), and coordinate projections; see \cref{rem:composition-set-operations}.
The following lemma records the corresponding recession-cone calculus with explicit proofs.
A key subtlety is that, for general convex sets, \emph{projection can enlarge the recession cone}; hence we state the strongest always-true inclusion and, separately, the equality that holds for polyhedra.

\begin{lemma}[Recession cones under product, affine intersection, and projection]\label{lem:composition-recession-polyhedral}
Let $C_1,C_2\subseteq\mathbb{R}^n$ be nonempty, closed, and convex.
Then:
\begin{enumerate}[label=(\roman*)]
    \item $0^+(C_1\times C_2)=0^+C_1\times 0^+C_2$.
    \item If $S \defeq \{x\mid Ax=b\}$ and $C\cap S\neq\varnothing$, then~$0^+(C\cap S)=0^+C\cap \ker A.$
    \item For any linear map $L$, one always has~$L(0^+C)\subseteq 0^+(L(C)).$
    If, in addition,~$C$ is a polyhedron, then~$0^+(L(C)) = L(0^+C)$, and in particular~$0^+(L(C))$ is polyhedral whenever~$0^+C$ is polyhedral.
\end{enumerate}
\end{lemma}

\begin{proof}
We repeatedly use the following characterization valid for nonempty closed convex sets $D$:
\begin{equation}\label{eq:one-point-recession}
d\in 0^+D
\ \Longleftrightarrow\
\exists \bar x\in D \ \text{s.t.}\ \bar x+td\in D\ \ \forall t\ge 0.
\end{equation}
The implication ``$\Rightarrow$'' is immediate.
For ``$\Leftarrow$'', fix $x\in D$ and $t\ge 0$.
Since $\bar x+ntd\in D$ for all $n\in\mathbb{N}$ and $D$ is convex,
\[
x_n\defeq \Bigl(1-\frac1n\Bigr)x+\frac1n(\bar x+ntd)
=x+td+\frac{\bar x-x}{n}\in D.
\]
By closedness, $x_n\to x+td$ implies $x+td\in D$, hence $d\in0^+D$.

\textbf{(i)}
Let $\tup{d_1,d_2}\in 0^+(C_1\times C_2)$.
Then for all $\tup{x_1,x_2}\in C_1\times C_2$ and $t\ge 0$,
$\tup{x_1+td_1,\ x_2+td_2}\in C_1\times C_2$,
so $x_1+td_1\in C_1$ and $x_2+td_2\in C_2$.
Thus $d_1\in 0^+C_1$ and $d_2\in 0^+C_2$, proving
$0^+(C_1\times C_2)\subseteq 0^+C_1\times 0^+C_2$.
Conversely, if $d_i\in 0^+C_i$ ($i=1,2$), then for any $\tup{x_1,x_2}\in C_1\times C_2$ and $t\ge 0$,
$\tup{x_1+td_1,\ x_2+td_2}\in C_1\times C_2$,
so $\tup{d_1,d_2}\in 0^+(C_1\times C_2)$.
Hence equality holds.

\textbf{(ii)}
Let $S=C\cap\{x\mid Ax=b\}\neq\varnothing$.
If $d\in 0^+S$, then by \eqref{eq:one-point-recession} there exists $\bar x\in S$ with $\bar x+td\in S$ for all $t\ge 0$.
Since $S\subseteq C$, we have $\bar x+td\in C$ for all $t\ge 0$, hence $d\in 0^+C$.
Moreover, $A(\bar x+td)=b$ and $A\bar x=b$ imply $Ad=0$, so $d\in\ker A$.
Thus $0^+S\subseteq 0^+C\cap\ker A$.
Conversely, if $d\in 0^+C\cap\ker A$, then for any $x\in S$ and $t\ge 0$,
$x+td\in C$ (since $d\in0^+C$) and
$A(x+td)=Ax+tAd=b$ (since $x\in S$ and $Ad=0$),
so $x+td\in S$.
Hence $d\in0^+S$ and $0^+C\cap\ker A\subseteq0^+S$, proving equality.

\textbf{(iii) inclusion.}
Let $d\in 0^+C$ and take any $y=Lx\in L(C)$.
Then for all $t\ge 0$,~$y+tL d = L(x+td)\in L(C)$,

so $Ld\in 0^+(L(C))$.
Thus $L(0^+C)\subseteq 0^+(L(C))$.

\textbf{(iii) equality for polyhedra.}
Assume $C$ is a polyhedron, and write it in $\mathbf{H}$-representation as
$C=\{x\mid Gx\ge h\}$.
Its recession cone is $0^+C=\{d\mid Gd\ge 0\}$.
Consider the image $L(C)=\{y\mid \exists x:\ y=Lx,\ Gx\ge h\}$, which is again a polyhedron.
A direction $w$ is in $0^+(L(C))$ if and only if for every feasible $\langle y,x\rangle$ with $y=Lx$, one has $y+tw\in L(C)$ for all $t\ge 0$, i.e., there exist $x_t$ with
$Lx_t = Lx+tw$ and $Gx_t\ge h$.
By polyhedrality, this is equivalent to the existence of some direction $d$ with $Gd\ge 0$ and $Ld=w$ (one can see this by applying the standard characterization of recession directions to the lifted polyhedron $\{\tup{x,y}\mid Gx\ge h,\ y=Lx\}$ and then projecting onto $y$).
Hence $w\in L(0^+C)$, proving $0^+(L(C))\subseteq L(0^+C)$, and thus equality.
\end{proof}

\begin{remark}[How DP interconnections map to set operations]\label{rem:composition-set-operations}
The three interconnections in \cref{def:dp-interconnections} are built from the operations of \cref{lem:composition-recession-polyhedral}:
\begin{enumerate}[label=(\alph*)]
    \item \emph{Parallel.} For feasible sets $\mathcal C_a\subseteq P\times Q$ and $\mathcal C_b\subseteq P'\times Q'$,
    $\mathcal C_{\mathrm{par}}=\mathcal C_a\times \mathcal C_b$.
    \item \emph{Series.} Writing the two intermediate variables separately as $q_a,q_b$,
    \begin{equation*}
    \begin{aligned}
    \widehat{\mathcal C}_{\mathrm{ser}}
    &\defeq
    (\mathcal C_a\times \mathcal C_b)\cap\{\tup{p,q_a,q_b,r}\mid q_a=q_b\},\\
    \mathcal C_{\mathrm{ser}}&=\Pi_{p,r}\!\left(\widehat{\mathcal C}_{\mathrm{ser}}\right).
    \end{aligned}
    \end{equation*}
    \item \emph{Feedback.} For $\mathcal C\subseteq (P\times R_{\mathrm{in}})\times(Q\times R_{\mathrm{out}})$,
    \begin{equation*}
    \begin{aligned}
    \widehat{\mathcal C}_{\mathrm{fb}}
    &\defeq
    \mathcal C\cap\{\tup{p,r_{\mathrm{in}},q,r_{\mathrm{out}}}\mid r_{\mathrm{in}}=r_{\mathrm{out}}\},\\
    \mathcal C_{\mathrm{fb}}&=\Pi_{p,q}\!\left(\widehat{\mathcal C}_{\mathrm{fb}}\right).
    \end{aligned}
    \end{equation*}
\end{enumerate}
Thus every finite composition graph built from parallel/series/feedback applies only these set operations.
\end{remark}

\section{Numerical Experiments}\label{sec:numerical-experiment}
We evaluate the computational pipeline developed in
\cref{sec:lcdp-computation} and its convex extension from
\cref{subsec:convex-epsilon}.
We consider three complementary benchmarks: (i) a synthetic series-chain family that isolates the trade-off between lifted sparsity and variable elimination for \emph{exact} \glspl{abk:lcdp}; (ii) a rigid gripper co-design instance that is natively polyhedral after an exact reparameterization (\emph{exact} \gls{abk:lcdp} regime); and (iii) a rover co-design instance with a bilinear coupling that yields a convex but nonpolyhedral feasible resource set, exercising the \emph{outer-approximation} pipeline.

Unless stated otherwise, all reported runtimes are end-to-end wall-clock times for a \emph{single} query instance, including model construction and solver calls but excluding plotting.\footnote{All experiments were run on a macOS Tahoe 26.2 machine with an Apple M2 Pro CPU (10 cores, 3.2\,GHz) and 16\,GB of RAM.}
All polyhedral queries are solved with \textsc{Bensolve}~\cite{lohne_vector_2017}.
As a baseline, we also compare against \textsc{MCDPL}~\cite{censi2015mathematical,censi2017uncertainty}, 
a general solver for \glspl{abk:dp}.

\begin{remark}[Discretization in \textsc{MCDPL}]\label{rem:mcdpl-discretization}
For any \gls{abk:dp} $\dprb$ whose antichains have infinite cardinality (e.g. a continuous relation such as $f_1 \leq r_1 + r_2$), \textsc{MCDPL} replaces the exact relation with a pair of finite-cardinality \glspl{abk:dp} $\langle\dprb_R^U , \dprb_R^L\rangle$ parameterized by a resolution $R \geq 1$.
The upper-bound \gls{abk:dp} $\dprb_R^U$ consists of $R$ points sampled on the Pareto boundary of the feasible resource set and serves as a pessimistic (inner) approximation of the true relation;
the lower-bound \gls{abk:dp} $\dprb_R^L$ is constructed by taking pairwise meets of $R+1$ successive boundary samples and serves as an optimistic (outer) approximation.
The solver uses Van~der~Corput sampling~\cite[Section 5.2]{lavalle2006planning}, so $\dprb_{R+1}^L \preceq \dprb_{R}^L$ and $\dprb_{R}^U \preceq \dprb_{R+1}^U$, guaranteeing \emph{monotone} convergence to the exact relation as $R \to \infty$.
\end{remark}

\subsection{Series-chain LCDP scalability benchmark}\label{subsec:series-chain-benchmark}
\subsubsection{Benchmark topology}
The benchmark (\cref{fig:series-chain}) is a series chain of~$m$ nodes~$\mathrm{DP}_0,\ldots,\mathrm{DP}_{m-1}$.
Each node has~$k$-dimensional functionality and~$k$-dimensional resource ports.
Consecutive nodes are interconnected in series by identifying the resource of~$\mathrm{DP}_i$ with the functionality of~$\mathrm{DP}_{i+1}$, componentwise.
The external ports are the functionality of~$\mathrm{DP}_0$, fixed to~$f_{\mathrm{ext}}=\mathbf{c}\in\mathbb{R}^k$, and the resource of~$\mathrm{DP}_{m-1}$, which is minimized in the posetal sense.
This acyclic topology isolates series composition and is therefore a natural stress test for the projection-based construction of \cref{subsec:compositional}.

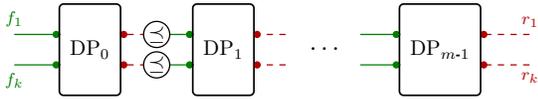
\begin{figure}[tb]
    \centering
    \scalebox{0.8}{\begin{tikzpicture}[DP, dp port sep=3.2]

        \node[dp={2}{2}] (dp0) {$\mathrm{DP}_0$};
        \node[dp={2}{2}, right=1.2cm of dp0] (dp1) {$\mathrm{DP}_1$};

        \node[right=0.8cm of dp1, font=\large] (dots) {$\cdots$};

        \node[dp={2}{2}, right=0.8cm of dots] (dpm) {$\mathrm{DP}_{m\text{-}1}$};

        \draw[rconn, rcname={}, fcname={}] (dp0_res1) to (dp1_fun1);
        \draw[rconn, rcname={}, fcname={}] (dp0_res2) to (dp1_fun2);

        \draw[dashed, dpred, semithick] (dp1_res1) -- ++(0.55cm,0);
        \draw[dashed, dpred, semithick] (dp1_res2) -- ++(0.55cm,0);

        \draw[dpgreen, semithick] ($(dots.east |- dpm_fun1)+(0.15cm,0)$) -- (dpm_fun1);
        \draw[dpgreen, semithick] ($(dots.east |- dpm_fun2)+(0.15cm,0)$) -- (dpm_fun2);
        \draw[funconn, relfun=above, funame={$\F{f_1}$}] (dp0_fun1);
        \draw[funconn, relfun=below, funame={$\F{f_k}$}] (dp0_fun2);
        \draw[runconn, relres=above, runame={$\R{r_1}$}] (dpm_res1);
        \draw[runconn, relres=below, runame={$\R{r_k}$}] (dpm_res2);
    \end{tikzpicture}}
    \caption{Series-chain benchmark with~$m$ nodes.
    Each~$\mathrm{DP}_i$ has~$k$-dimensional functionality and resource ports, and the resource of~$\mathrm{DP}_i$ feeds the functionality of~$\mathrm{DP}_{i+1}$.}\label{fig:series-chain}
    \vspace{-2mm}
\end{figure}

\subsubsection{Component generation}
Each node is a synthetic \gls{abk:ldp} with~$k$ functionality variables~$f\in\mathbb{R}_{\ge 0}^k$,~$k$ resource variables~$r\in\mathbb{R}_{\ge 0}^k$, and~$c\defeq k+1$ linear inequalities.
The first~$k$ rows are \emph{pass-through} constraints
\begin{equation}\label{eq:series-pass}
    f_j \le \gamma_j r_j + d_j,\qquad j=1,\ldots,k,
\end{equation}
with~$\gamma_j \sim \mathrm{Unif}(0.5,1.5)$.
The remaining row is a \emph{cross-coupling trade-off} constraint
\begin{equation}\label{eq:series-trade}
    a\,f_{j^\star} \le \sum_{i=1}^{k} b_i r_i + e,
\end{equation}
with~$a,b_i\sim \mathrm{Unif}(0.5,2.0)$ and fixed~$j^\star=1$.
This produces nontrivial Pareto frontiers while keeping the per-node complexity fixed as~$m$ varies.
All random coefficients are drawn with a fixed seed.
Intercepts are chosen with a feasibility margin~$\mu=0.1$, ensuring that the tested query is feasible.

We further require that each inequality contain at most one functionality variable on the left-hand side.
This matches the most compact \textsc{MCDPL} encoding used here; more general left-hand-side couplings can be represented by auxiliary compositions, but substantially increase the symbolic size of the baseline model and would confound the scaling comparison.

\subsubsection{Compared methods and parameters}
We compare three methods.
First, in the \emph{monolitic (exact)}, we assemble the lifted block-angular polyhedron of \cref{subsec:monolithic} and solve the resulting \gls{abk:molp} exactly with \textsc{Bensolve}~\cite{lohne_vector_2017}.
Second, using \emph{hybrid compositional elimination (exact)}, we apply series contractions using \gls{abk:fme} with redundancy removal; if the intermediate inequality count exceeds a threshold~$\bar c$, we fall back to the monolithic representation for the remaining subsystem.
Finally, with \emph{MCDPL (approximate)}, encode the same chain in the current state of the art \textsc{MCDPL} and compute optimistic/pessimistic Pareto approximations at resolution~$R$~\cite{censi2017uncertainty, zardini2023co}.
Unless stated otherwise, we use~$k=2$,~$c=k+1=3$ constraints per node, seed~$=42$, FME cutoff~$\bar c=50$, \textsc{MCDPL} resolution~$R=5$, and timeout~$300\,\mathrm{s}$.

\paragraph{Results}
For a chain of length $m$, the monolithic model has~$n_{\mathrm{vars}} = 2km$,~$n_{\mathrm{ineq}} = cm + 2k(m-1)$, since each node contributes~$2k$ port variables and~$c$ local inequalities, and each of the~$m-1$ interconnections contributes~$k$ equalities written as two inequalities.
For~$k=2$ and~$c=3$, this gives~$n_{\mathrm{vars}}=4m$ and~$n_{\mathrm{ineq}}=7m-4$, matching \cref{tab:sweep}.

\begin{table}[t]
    \caption{Series-chain benchmark ($k{=}2$, $c{=}3$, fixed $f_{\mathrm{ext}}{=}\mathbf{1}$).
    ``Vars'' and ``Cstr'' count scalar decision variables and inequalities in the feasible-set representation.
    ``$\checkmark$'' denotes exact agreement between the monolithic and hybrid pipelines in Pareto vertex/ray count.}\label{tab:sweep}
    \centering\small
    \setlength{\tabcolsep}{3pt}
    \renewcommand{\arraystretch}{0.9}
    \begin{adjustbox}{max width=0.96\columnwidth, center}
    \begin{tabular}{r|rrr|rrr|rc|c}
        \toprule
        & \multicolumn{3}{c|}{Monolithic} & \multicolumn{3}{c|}{Hybrid FME} & \multicolumn{2}{c|}{MCDPL} & \\
        $m$ & Vars & Cstr & Time & Vars & Cstr & Time & Time & Pts & Agree \\
        \midrule
        3  & 12  & 17  & 0.018 & 4  & 12  & 0.328  & 11.1  & 27 & $\checkmark$ \\
        4  & 16  & 24  & 0.011 & 4  & 14  & 0.584  & 11.2  & 34 & $\checkmark$ \\
        6  & 24  & 38  & 0.010 & 4  & 26  & 1.680  & 47.2  & 39 & $\checkmark$ \\
        8  & 32  & 52  & 0.024 & 4  & 20  & 3.083  & 19.0  & 22 & $\checkmark$ \\
        10 & 40  & 66  & 0.020 & 4  & 29  & 5.598  & 27.0  & 44 & $\checkmark$ \\
        15 & 60  & 101 & 0.031 & 4  & 36  & 31.697 & 43.5  & 44 & $\checkmark$ \\
        20 & 80  & 136 & 0.044 & 4  & 40  & 61.240 & 67.3  & 44 & $\checkmark$ \\
        30 & 120 & 206 & 0.090 & 64 & 151 & 33.612 & 127.0 & 44 & $\checkmark$ \\
        \bottomrule
    \end{tabular}
    \end{adjustbox}
    \vspace{-2mm}
\end{table}

The monolithic pipeline is consistently fastest, with end-to-end time essentially constant at the scale tested.
Its advantage here comes from preserving the sparse block-angular structure of the lifted representation and avoiding any projection step.
The hybrid \gls{abk:fme} pipeline eliminates all internal ports and reduces the final query dimension to the $2k$ external variables, but its preprocessing cost grows with $m$ due to repeated projection and redundancy removal.
The non-monotone final constraint counts in \cref{tab:sweep} reflect the fact that many generated inequalities are subsequently removed as redundant.
At the chosen cutoff~$\bar c=50$, all reported instances remained fully compositional; the ``hybrid'' safeguard is included for robustness on larger chains.
At fixed resolution~$R=10$, \textsc{MCDPL} is one to three orders of magnitude slower than the algebraic pipelines and returns resolution-dependent Pareto point sets rather than the exact polyhedral frontier.
For all tested~$m$, the monolithic and hybrid pipelines return identical Pareto vertex/ray counts, confirming that the compositional elimination preserves the exact feasible set (\cref{prop:compositional-correct}).

\subsection{Rigid gripper co-design: exact polyhedral regime}\label{subsec:gripper-case-study}
The series-chain benchmark isolates algorithmic scaling on synthetic data.
We next consider a small but physically grounded multi-objective co-design problem that is an \emph{exact} \gls{abk:lcdp} after suitable reparameterization, characterizing the zero-approximation-error regime of the proposed framework.

\subsubsection{System and query}
We study a tendon-driven rigid gripper with four conceptual design sub-problems: a motor (maps actuator investment $\tup{m_{\mathrm{mot}},c_{\mathrm{mot}}}$ to achievable tendon force~$f_{\mathrm{grasp}}$), a material-selection module (given a mass and cost budget~$\tup{m_{\mathrm{f}},c_{\mathrm{f}}}$, determines achievable segment lengths~$\tup{L_{\mathrm{Al}},L_{\mathrm{CF}}}$), a gripper-geometry module (maps lengths~$\tup{L_{\mathrm{Al}},L_{\mathrm{CF}}}$ to structural workspace and grasping capacity), and an arm-cost aggregation module (maps workspace requirement $f_{\mathrm{ws}}$ to arm mounting cost~$c_{\mathrm{arm}}$).
For a co-design diagram, refer to \cref{fig:gripper-architecture}.
The system-level functionality is~$
q=\tup{f_{\mathrm{grasp}},f_{\mathrm{ws}}}$, 
where~$f_{\mathrm{grasp}}$ is the required grasping force and~$f_{\mathrm{ws}}$ is the required workspace.
The system-level resources are~$r=\tup{r_{\mathrm{cost}},r_{\mathrm{mass}}}$, which are minimized in the Pareto sense.
We use the query~$f_{\mathrm{grasp}}\ge 100~\mathrm{N}$,~$f_{\mathrm{ws}}\ge 80~\mathrm{mm}$.

\subsubsection{LCDP model formulation}\label{subsubsec:gripper-formulation}
The gripper \gls{abk:lcdp} is parameterized by two continuous design variables,~$\tup{L_{\mathrm{Al}},L_{\mathrm{CF}}}\in\mathbb{R}_{\ge 0}^2$, representing the aluminum and carbon-fibre finger lengths.
These variables are the main coupling medium across the architecture, as they jointly determine workspace, structural feasibility, finger mass, and material cost.
The resulting model has~$n_{\mathrm{f}}=2$ functionality dimensions,~$n_{\mathrm{r}}=8$ resource dimensions, and seven linear inequalities grouped by subsystem.

\begin{subequations}\label{eq:gripper-lcdp}

\paragraph{Actuation (motor sub-problem)}
The available grasp force increases with both motor mass and motor cost.
A heavier motor accommodates a larger rotor, while a costlier motor can achieve higher torque density.
We capture these two effects with the supporting halfspace
\begin{equation}\label{eq:gripper-motor}
    f_{\mathrm{grasp}} \le \alpha_{\mathrm{f}}\, m_{\mathrm{mot}} + \beta_{\mathrm{f}}\, c_{\mathrm{mot}},
\end{equation}
where~$\alpha_{\mathrm{f}}=400\,\mathrm{N/kg}$ and~$\beta_{\mathrm{f}}=0.6\,\mathrm{N/USD}$.

\paragraph{Kinematics and structure (finger sub-problem)}
Finger length determines both reach and structural feasibility.
The usable workspace grows linearly with total finger length,
\begin{equation}\label{eq:gripper-ws}
    f_{\mathrm{ws}} \le k_{\mathrm{w}} \bigl(L_{\mathrm{Al}} + L_{\mathrm{CF}}\bigr),
    \qquad k_{\mathrm{w}}=0.35,
\end{equation}
while sustaining a grasping force requires a minimum total length,
\begin{equation}\label{eq:gripper-struct}
    L_{\mathrm{Al}} + L_{\mathrm{CF}} \ge \gamma_{\mathrm{L}}\, f_{\mathrm{grasp}},
    \qquad \gamma_{\mathrm{L}}=0.5\,\mathrm{mm/N}.
\end{equation}
Thus both functionality requirements couple to the same design pair~$\tup{L_{\mathrm{Al}},L_{\mathrm{CF}}}$.

\paragraph{Material physics (material sub-problem)}
The finger lengths also determine mass and raw material cost:
\begin{equation}\label{eq:gripper-fmass}
    m_{\mathrm{f}} \ge \rho_{\mathrm{Al}}\, L_{\mathrm{Al}} + \rho_{\mathrm{CF}}\, L_{\mathrm{CF}},
\end{equation}
\begin{equation}\label{eq:gripper-fcost}
    c_{\mathrm{f}} \ge p_{\mathrm{Al}}\, L_{\mathrm{Al}} + p_{\mathrm{CF}}\, L_{\mathrm{CF}}.
\end{equation}
We use
$\rho_{\mathrm{Al}}=0.012\,\mathrm{kg/mm}$,
$\rho_{\mathrm{CF}}=0.004\,\mathrm{kg/mm}$,
$p_{\mathrm{Al}}=2.0\,\mathrm{USD/mm}$, and
$p_{\mathrm{CF}}=8.0\,\mathrm{USD/mm}$.
Hence carbon fibre is lighter but more expensive per unit length, which is the main source of the Pareto trade-off.

\paragraph{System-level aggregation}
The system-level objectives aggregate subsystem resources.
Total mass is
\begin{equation}\label{eq:gripper-mass}
    r_{\mathrm{mass}} \ge m_{\mathrm{mot}} + m_{\mathrm{f}},
\end{equation}
and total cost combines motor cost, arm payload penalties, finger material cost, and a workspace-dependent term:
\begin{equation}\label{eq:gripper-cost}
    r_{\mathrm{cost}} \ge c_{\mathrm{mot}} + k_{\mathrm{AP}}\, m_{\mathrm{mot}}
        + c_{\mathrm{f}} + k_{\mathrm{AP}}\, m_{\mathrm{f}}
        + k_{\mathrm{AW}}\, f_{\mathrm{ws}} + c_{\mathrm{A0}},
\end{equation}
with
$k_{\mathrm{AP}}=200\,\mathrm{USD/kg}$,
$k_{\mathrm{AW}}=50\,\mathrm{USD/mm}$, and
$c_{\mathrm{A0}}=300\,\mathrm{USD}$.
Substituting \eqref{eq:gripper-fmass}--\eqref{eq:gripper-fcost} into \eqref{eq:gripper-cost} yields effective per-length coefficients~$p_{\mathrm{Al}} + k_{\mathrm{AP}}\rho_{\mathrm{Al}} = 4.4\,\mathrm{USD/mm}$,~$p_{\mathrm{CF}} + k_{\mathrm{AP}}\rho_{\mathrm{CF}} = 8.8\,\mathrm{USD/mm}$, which compactly combine material price and arm load-bearing penalty.
The density--cost asymmetry between aluminum and carbon fiber drives the Pareto trade-off between total mass and total cost.
\end{subequations}

\begin{figure}[tb]
    \centering
    \includegraphics[width=\columnwidth]{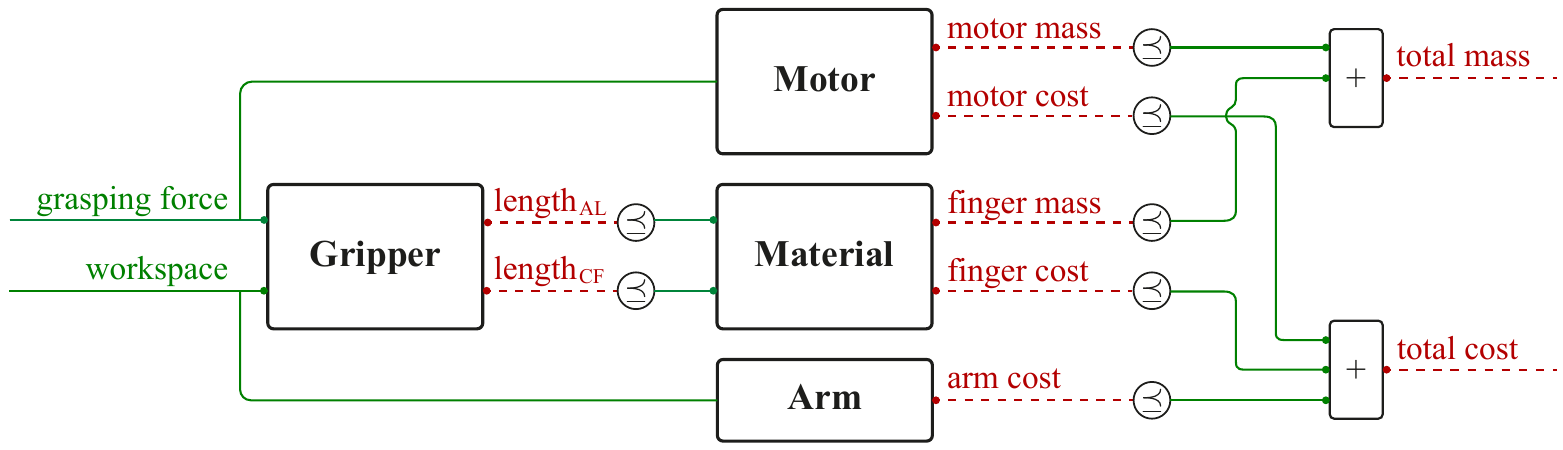}
    \caption{
    Gripper co-design architecture.
    The query~$\tup{f_{\mathrm{grasp}},f_{\mathrm{ws}}}$ is propagated to subsystem design choices; the motor and material subsystems are coupled through the shared design variables~$\tup{L_{\mathrm{Al}},L_{\mathrm{CF}}}$; and contributions are aggregated into system objectives $\tup{r_{\mathrm{cost}},r_{\mathrm{mass}}}$.}
    \label{fig:gripper-architecture}
\end{figure}

\subsubsection{Results}
We compare: (i) the \emph{monolithic exact} \gls{abk:lcdp} pipeline, which solves the induced bi-objective \gls{abk:molp} with \textsc{Bensolve}; and (ii) \textsc{MCDPL} optimistic approximations at resolutions~$R\in\{5,10,20,40,60,85\}$.
Because the model is already polyhedral, this benchmark cleanly separates the exact \gls{abk:lcdp} regime from discretization-based approximation.

\begin{figure*}[tb]
    \centering
    \includegraphics[width=\textwidth]{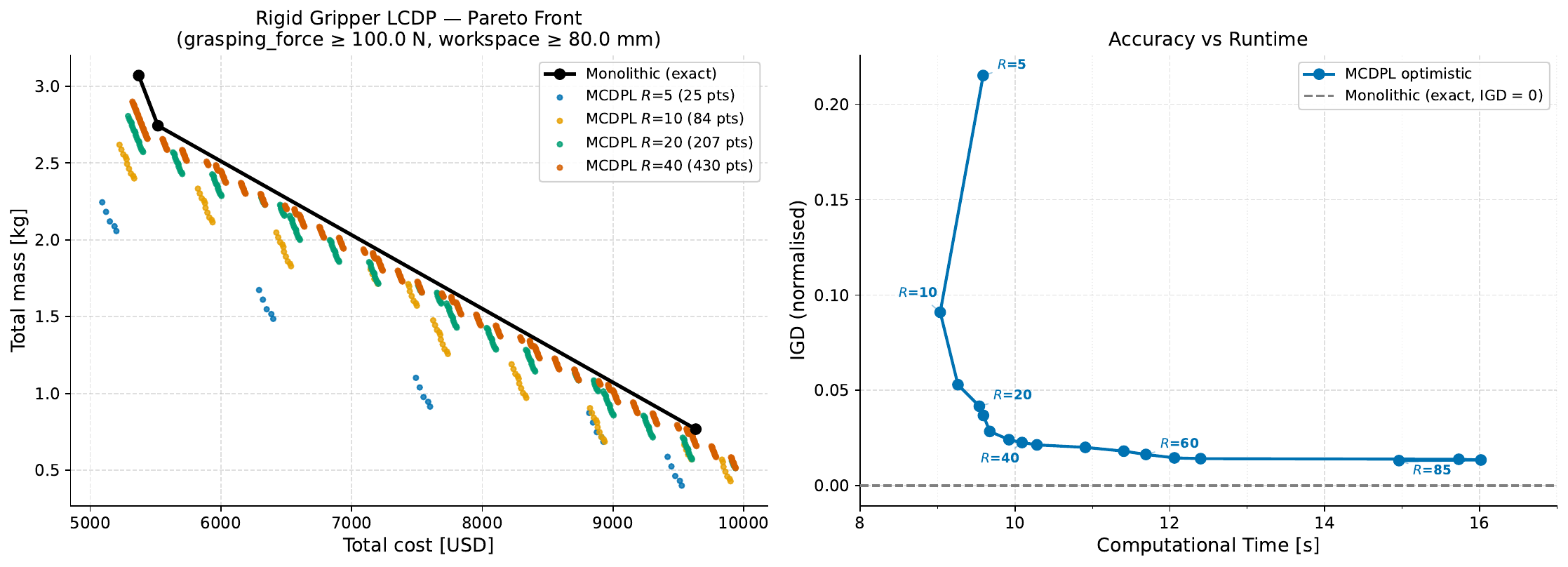}
    \caption{Pareto-front comparison for the rigid gripper case study.
    (Left)~Monolithic exact frontier (3 vertices) versus \textsc{MCDPL} optimistic approximations at resolutions $R\in\{5,10,20,40\}$.
    (Right)~IGD convergence versus computational time: the normalised Inverted Generational Distance of \textsc{MCDPL} optimistic solutions decreases monotonically with increasing resolution, while the exact monolithic solver achieves IGD\,$=$\,0.}
    \label{fig:gripper-pareto}
\end{figure*}

\Cref{tab:gripper-pareto} reports the three exact Pareto vertices returned by the monolithic solver; \cref{tab:gripper-comparison} summarizes timing and approximation quality.
The right panel of \cref{fig:gripper-pareto} visualises the accuracy--runtime trade-off.

\begin{table}[tb]
    \caption{Exact Pareto vertices for the rigid gripper \gls{abk:lcdp} (monolithic solver).
    Query: $f_{\mathrm{grasp}}\ge 100\,\mathrm{N}$,\; $f_{\mathrm{ws}}\ge 80\,\mathrm{mm}$.
    $V_1$--$V_3$ are the only Pareto-optimal solutions; no approximation is introduced.}\label{tab:gripper-pareto}
    \centering
    \renewcommand{\arraystretch}{0.9}
    \begin{adjustbox}{max width=0.96\columnwidth, center}
    \begin{tabular}{clcc}
        \toprule
        Vertex & Dominant material mix & $r_{\mathrm{cost}}$ (USD) & $r_{\mathrm{mass}}$ (kg) \\
        \midrule
        $V_1$ & Pure aluminum\;\;(min cost / max mass) & 5\,371.06 & 3.0696 \\
        $V_2$ & Mixed Al\,/\,CF\;\;(balanced trade-off) & 5\,518.10 & 2.7429 \\
        $V_3$ & Pure carbon fibre (min mass / max cost) & 9\,633.00 & 0.7677 \\
        \bottomrule
    \end{tabular}
    \end{adjustbox}
    \vspace{-2mm}
\end{table}

\begin{table}[tb]
    \caption{Rigid gripper: monolithic exact solver vs.\ \textsc{MCDPL} optimistic approximations.
    IGD is the Inverted Generational Distance computed against the densified exact frontier (200 reference points), normalised to~$[0,1]^2$ using the ideal--nadir range.}\label{tab:gripper-comparison}
    \centering
    \renewcommand{\arraystretch}{0.8}
    \begin{adjustbox}{max width=1\columnwidth, center}
    \begin{tabular}{lrrcc}
        \toprule
        Method & $R$ & \#Pts & Time (s) & IGD \\
        \midrule
        Monolithic (exact)  & ---  &    3 & $4.8\times10^{-2}$ & 0 \\
        \textsc{MCDPL} opt. &   5  &   25 & $9.59$             & $2.15\times10^{-1}$ \\
        \textsc{MCDPL} opt. &  10  &   84 & $9.03$             & $9.09\times10^{-2}$ \\
        \textsc{MCDPL} opt. &  20  &  207 & $9.54$             & $4.16\times10^{-2}$ \\
        \textsc{MCDPL} opt. &  40  &  430 & $10.09$            & $2.26\times10^{-2}$ \\
        \textsc{MCDPL} opt. &  60  &  642 & $11.69$            & $1.64\times10^{-2}$ \\
        \textsc{MCDPL} opt. &  85  &  944 & $14.96$            & $1.31\times10^{-2}$ \\
        \bottomrule
    \end{tabular}
    \end{adjustbox}
    \vspace{-2mm}
\end{table}

The monolithic \gls{abk:lcdp} solve enumerates the exact frontier (3 vertices) in approximately~$48$ ms end-to-end.
In contrast, \textsc{MCDPL} optimistic runs require approximately~$9$--$15$ s per resolution and produce increasingly dense discretized fronts (from 25 points at $R=5$ to 944 points at $R=85$).
Approximation quality is measured by the normalised Inverted Generational Distance (IGD), computed against a densified reference front (200 points interpolated along the exact piecewise-linear Pareto curve) and normalised to~$[0,1]^2$ using the ideal--nadir range.
The IGD decreases monotonically from~$2.15\times 10^{-1}$ at~$R=5$ to~$1.31\times 10^{-2}$ at~$R=85$, confirming theoretical convergence of the optimistic approximation.
Even at~$R=85$, the exact monolithic pipeline remains about~$300\times$ faster while delivering zero approximation error.

\subsection{Rover co-design case study with bilinear coupling}\label{subsec:rover}
The rigid gripper benchmark illustrates the \emph{exact} polyhedral regime.
We now consider a small but representative example that falls outside the exact \gls{abk:lcdp} class, i.e., a rover co-design problem with a bilinear battery coupling.
For the fixed query studied here, the induced feasible resource set is closed and convex but not polyhedral, making it a natural test case for the outer-approximation framework of \cref{subsec:convex-epsilon}.

\begin{figure}[tb]
    \centering
    \includegraphics[width=\columnwidth]{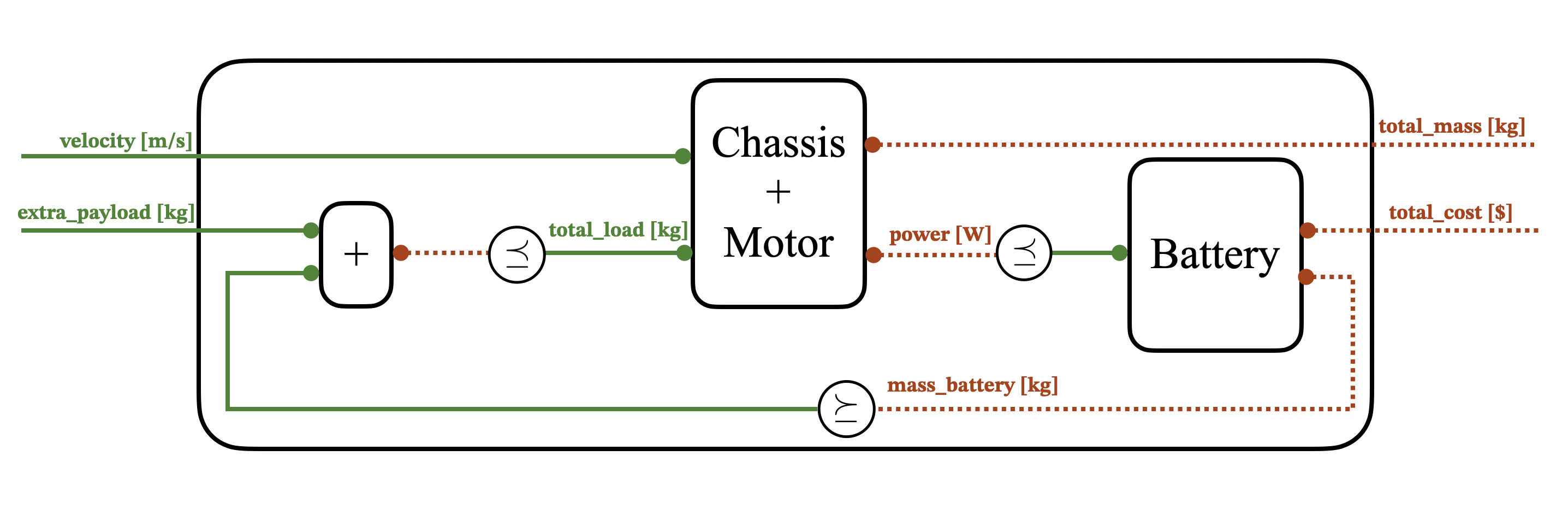}
    \caption{Rover co-design architecture. The query $\tup{m_{\mathrm{pay}},v_{\mathrm{req}}}$ propagates to chassis and battery choices; battery and chassis are coupled through power and mass; and the system-level objectives are $\tup{c_{\mathrm{sys}},m_{\mathrm{sys}}}$.}
    \label{fig:rover-codesign-architecture}
    \vspace{-3mm}
\end{figure}

\subsubsection{Model}
The rover architecture in \cref{fig:rover-codesign-architecture} couples a chassis subsystem and a battery subsystem.
The system-level functionality query is~$f=\tup{m_{\mathrm{pay}},v_{\mathrm{req}}}$, and the system-level resource vector to be minimized is~$r=\tup{c_{\mathrm{sys}},m_{\mathrm{sys}}}$.\footnote{In \cref{fig:rover-codesign-architecture}, these correspond to $\tup{\texttt{extra\_payload},\texttt{velocity}}$ and $\tup{\texttt{overall\_cost},\texttt{total\_mass}}$, respectively.}

\paragraph{Chassis}
The chassis provides total load $\ell$ and velocity $v$, and requires power $p_{\mathrm{ch}}$ and total mass $m_{\mathrm{tot}}$:
\begin{equation}\label{eq:rover-chassis}
p_{\mathrm{ch}} \ge p_0 + \kappa(\ell + m_{\mathrm{s}})v, \quad m_{\mathrm{tot}} \ge \ell + m_{\mathrm{s}},
\end{equation}
with $m_{\mathrm{s}}=770~\mathrm{kg}$, $p_0=10~\mathrm{W}$, and
$\kappa=1~\mathrm{W\cdot s/(m\cdot kg)}$.

\paragraph{Battery}
The battery provides power $p_{\mathrm{bat}}$ and requires battery mass $m_{\mathrm{b}}$ and cost $c$:
\begin{equation}\label{eq:rover-battery}
p_{\mathrm{bat}} \le \alpha m_{\mathrm{b}} + \beta c + \gamma m_{\mathrm{b}} c,
\end{equation}
with $\alpha=0.975~\mathrm{W/kg}$, $\beta=0.02733~\mathrm{W/USD}$, and $\gamma=0.0005~\mathrm{W/(kg\cdot USD)}$.
The bilinear term $\gamma m_{\mathrm{b}} c$ is the sole source of nonpolyhedrality.

\paragraph{Interconnection}
The two subsystems are wired through
\begin{equation}\label{eq:rover-wiring}
\begin{aligned}
m_{\mathrm{pay}}+m_{\mathrm{b}} \le \ell,\quad p_{\mathrm{ch}} \le p_{\mathrm{bat}},\quad v_{\mathrm{req}} \le v,\\
c_{\mathrm{sys}} \ge c,\quad m_{\mathrm{sys}} \ge m_{\mathrm{tot}}.
\end{aligned}
\end{equation}

We use the query
\(
m_{\mathrm{pay}}=0.1~\mathrm{kg}
\)
and
\(
v_{\mathrm{req}}=0.05~\mathrm{m/s}
\),
and minimize the bi-objective resource vector
\(
\tup{c_{\mathrm{sys}},m_{\mathrm{sys}}}
\).

\subsubsection{Compared methods}
We compare four solution routes.

\paragraph{Method I: high-resolution reference frontier}
We densely sample the objective plane $\tup{c_{\mathrm{sys}},m_{\mathrm{sys}}}$, test nonlinear feasibility of \eqref{eq:rover-chassis}--\eqref{eq:rover-wiring}, and retain nondominated feasible samples.
This produces a high-resolution reference frontier used for error reporting.

\paragraph{Method II: \textsc{MCDPL}}
We solve the same query with \textsc{MCDPL} at resolutions $R\in\{5,10,20,40,60,80\}$, obtaining optimistic and pessimistic Pareto approximations.

\paragraph{Method III: tangent-plane outer approximation + monolithic \gls{abk:molp}}
For the fixed query, the required chassis power is
    \begin{equation}\label{eq:preq}
        p_{\mathrm{req}}(m_{\mathrm{b}})=p_0+\kappa\bigl(m_{\mathrm{s}}+m_{\mathrm{pay}}+m_{\mathrm{b}}\bigr)\,v_{\mathrm{req}}.
    \end{equation}
Combining \eqref{eq:rover-battery} with $p_{\mathrm{ch}}\le p_{\mathrm{bat}}$ yields
    \begin{equation}\label{eq:phi}
        c \ge \phi(m_{\mathrm{b}})
        \defeq
        \frac{p_{\mathrm{req}}(m_{\mathrm{b}})-\alpha m_{\mathrm{b}}}{\beta+\gamma m_{\mathrm{b}}}.
\end{equation}
Writing $\phi(m_{\mathrm{b}})=(a+b m_{\mathrm{b}})/(\beta+\gamma m_{\mathrm{b}})$ with
    $a=p_0+\kappa(m_{\mathrm{s}}+m_{\mathrm{pay}})v_{\mathrm{req}}$ and
    $b=\kappa\,v_{\mathrm{req}}-\alpha$, one verifies that
    \[
    \phi''(m_{\mathrm{b}})
    =
    \frac{2\gamma\bigl(\gamma a-\beta b\bigr)}{(\beta+\gamma m_{\mathrm{b}})^3}
    >0
    \]
    for the parameter values; hence~$\phi$ is convex on~$m_{\mathrm{b}}\ge 0$.
    We select anchor points~$\{m_{\mathrm{b}}^{(i)}\}_{i=1}^N$ and tangent cuts
    \begin{equation}\label{eq:rover-tangent-cuts}
        c \ge \phi(m_{\mathrm{b}}^{(i)}) + \phi'(m_{\mathrm{b}}^{(i)})(m_{\mathrm{b}}-m_{\mathrm{b}}^{(i)}),
        \qquad i=1,\ldots,N.
    \end{equation}
    Because tangents under-estimate a convex function, \eqref{eq:rover-tangent-cuts} defines a polyhedral \emph{outer} approximation of the feasible set.
    Together with the linear chassis and wiring constraints, this yields an $N$-dependent \gls{abk:lcdp} solved via the monolithic \gls{abk:molp} pipeline.

\paragraph{Method IV: direct cutting-plane refinement in resource space}
Starting from a coarse outer polyhedron in $\tup{c_{\mathrm{sys}},m_{\mathrm{sys}}}$, we iteratively add supporting half-space cuts at boundary points and re-solve the resulting \gls{abk:molp}.
This produces a sequence of polyhedral outer approximations in resource space.

\subsubsection{Metrics}
We report both efficiency and accuracy.
Efficiency is measured by end-to-end runtime and, for Methods III--IV, the number of inequalities in the polyhedral surrogate.
Accuracy is measured against the reference frontier using directional one-sided excesses on a fixed normalization box induced by the reference frontier.
For outer approximations, we report~$e[P,C]$ and~$e[C,P]$.
For inner approximations, the primary directional metric is~$e[C,P^-]$ (coverage of the reference by the inner front), rather than~$e[P^-,C]$.
To reduce discretization-induced oscillations and obtain a scale-aware indicator for inner fronts, we additionally report the normalized hypervolume deficit
\[
\Delta HV_{\mathrm{rel}}^-
\;\defeq\;
\frac{HV(C)-HV(P^-)}{HV(C)}.
\]
For optimistic fronts we report the dual hypervolume excess
\(
\Delta HV_{\mathrm{rel}}^+ = (HV(P^+)-HV(C))/HV(C)
\).
For Methods III--IV, the recession-cone discrepancy is $\delta=0$ by \cref{thm:epsilon-linear-approx}.
We also report the maximum and mean mass gap (kg) for physical interpretability.

\subsubsection{Results}
\Cref{fig:rover-pareto-comparison,fig:rover-convergence} and \cref{tab:rover-method-comparison} summarize the rover results.
Both \gls{abk:ldp}-based outer-approximation pipelines improve monotonically as the number of cuts grows, consistent with the $\tup{\varepsilon,0}$ theory of \cref{subsec:convex-epsilon}.
For Method III (tangent-plane surrogate), at $N=200$ we obtain~$e[P,C] = 2.36\times 10^{-4}$,~$e[C,P] = 5.37\times 10^{-3}$,~$\Delta HV_{\mathrm{rel}}^- = 3.75\times 10^{-3}$

with end-to-end runtime $2.53\times 10^{-2}\,\mathrm{s}$.
Method IV reaches a comparable vertex excess at $N=200$ ($e[P,C]=2.37\times 10^{-4}$), but has larger coverage excess ($e[C,P]=1.33\times 10^{-2}$) and is much slower ($3.38\times 10^{1}\,\mathrm{s}$) because each refinement step requires an additional \gls{abk:molp} solve.

Under the new inner-approximation metrics, MCDPL pessimistic fronts now exhibit the expected smooth directional improvement:
$e[C,P^-]$ decreases from $3.74\times 10^{-1}$ at $R=5$ to $1.13\times 10^{-1}$ at $R=80$, and
$\Delta HV_{\mathrm{rel}}^-$ decreases from $4.93\times 10^{-1}$ to $1.54\times 10^{-1}$.
However, even at the best tested MCDPL setting ($R=80$), the inner-front errors remain far above Method III while requiring $2.45\times 10^{3}\,\mathrm{s}$ of solve time.
Thus, the LDP outer-approximation pipeline still offers the strongest accuracy--time trade-off by a wide margin.

\begin{figure*}[tb]
    \centering
    \includegraphics[width=\textwidth]{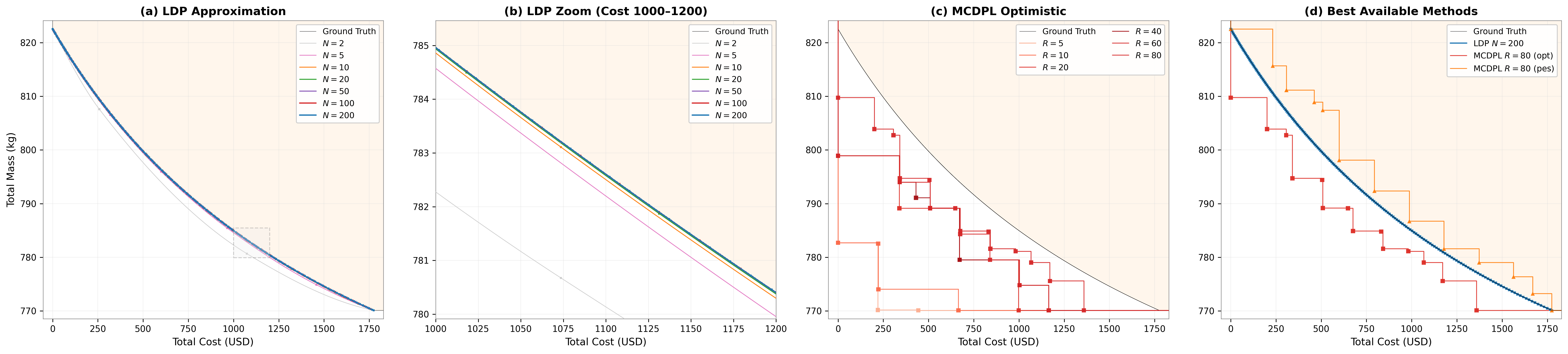}
    \caption{Pareto-front comparison for the rover case study.
    (a)~LDP tangent-plane approximation at several $N$ vs.\ ground truth;
    (b)~zoomed view highlighting convergence in the cost range 1000--1200~USD;
    (c)~MCDPL optimistic approximations at resolutions $R\in\{5,10,20,40,60,80\}$;
    (d)~best available result from each method compared against the ground-truth frontier.}
    \label{fig:rover-pareto-comparison}
    \vspace{-2mm}
\end{figure*}

\begin{figure*}[tb]
    \centering
    \includegraphics[width=\textwidth]{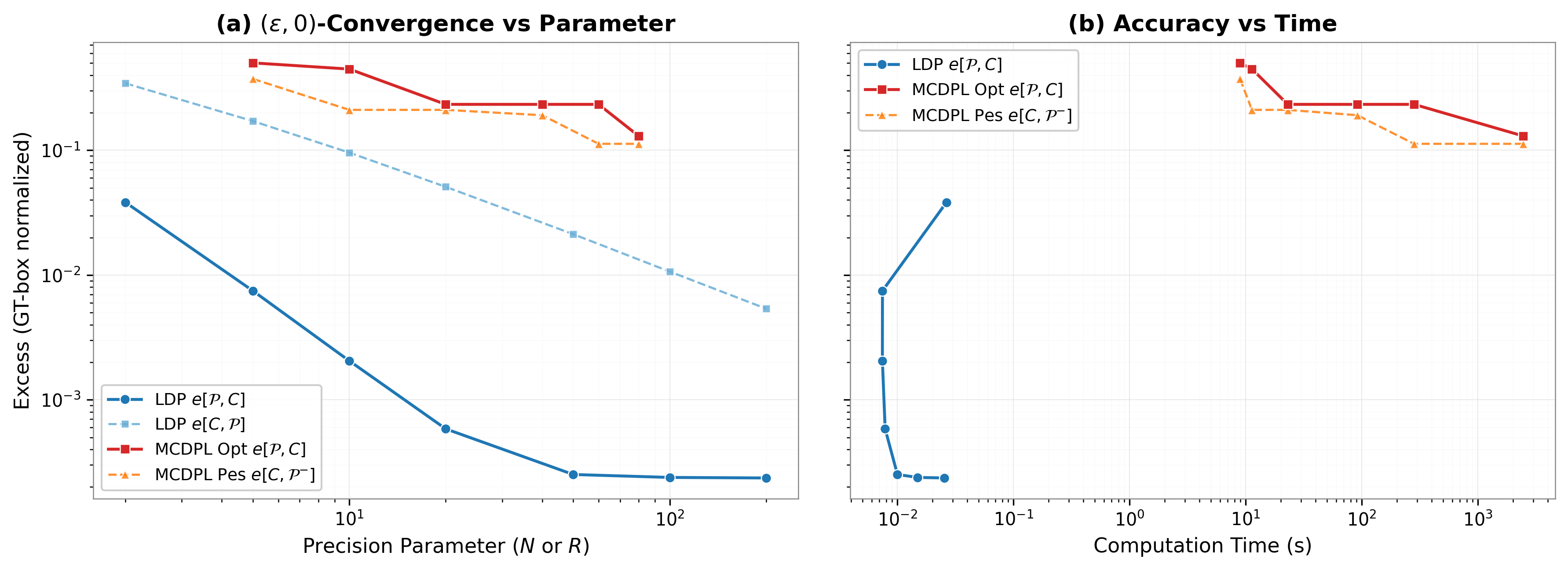}
    \caption{Convergence and efficiency comparison.
    (a)~Directional excess metrics versus precision parameter: $e[\mathcal{P},C]$ and $e[C,\mathcal{P}]$ for LDP, $e[\mathcal{P}^+,C]$ for MCDPL optimistic, and $e[C,\mathcal{P}^-]$ for MCDPL pessimistic;
    (b)~the same directional excess metrics plotted against computation time, showing that LDP achieves markedly lower error at a fraction of the MCDPL runtime.}
    \label{fig:rover-convergence}
    \vspace{-2mm}
\end{figure*}

\begin{table*}[tb]
  \caption{Rover case study: accuracy--time comparison across methods.
  For Methods~III--IV, $\delta=0$ holds by \cref{thm:epsilon-linear-approx}.
  The column $\Delta HV_{\mathrm{rel}}$ reports a one-sided normalized hypervolume gap with respect to the reference frontier $C$:
  hypervolume \emph{deficit} $(HV(C)-HV(P^-))/HV(C)$ for inner fronts and hypervolume \emph{excess} $(HV(P^+)-HV(C))/HV(C)$ for optimistic fronts.
  For MCDPL pessimistic rows, the coverage column corresponds to $e[C,\mathcal{P}^{-}]$.}
  \label{tab:rover-method-comparison}
  \centering
  \renewcommand{\arraystretch}{0.8}
  \begin{adjustbox}{max width=0.98\textwidth, center}
  \setlength{\tabcolsep}{2.8pt}
  \begin{tabular*}{\textwidth}{@{\extracolsep{\fill}}llcccccccc}
    \toprule
    Method & Parameter & \(n_{\mathrm{ineq}}\) & Pts & \(e[\mathcal{P},C]\) & \(e[C,\mathcal{P}]\) & \(\Delta HV_{\mathrm{rel}}\) & Max gap (kg) & Mean gap (kg) & Time (s) \\
    \midrule
    \multirow{7}{*}{\shortstack[l]{Tangent-plane\\linearization\\(Method~III)}}
      & \(N=2\)   & 11  & 4   & \(3.83\times 10^{-2}\) & \(3.44\times 10^{-1}\) & \(2.69\times 10^{-1}\) & \(2.53\) & \(7.59\times 10^{-1}\) & \(2.65\times 10^{-2}\) \\
      & \(N=5\)   & 14  & 7   & \(7.47\times 10^{-3}\) & \(1.71\times 10^{-1}\) & \(1.40\times 10^{-1}\) & \(4.59\times 10^{-1}\) & \(1.30\times 10^{-1}\) & \(7.44\times 10^{-3}\) \\
      & \(N=10\)  & 19  & 12  & \(2.05\times 10^{-3}\) & \(9.57\times 10^{-2}\) & \(7.57\times 10^{-2}\) & \(1.22\times 10^{-1}\) & \(3.60\times 10^{-2}\) & \(7.43\times 10^{-3}\) \\
      & \(N=20\)  & 29  & 22  & \(5.84\times 10^{-4}\) & \(5.09\times 10^{-2}\) & \(3.91\times 10^{-2}\) & \(3.40\times 10^{-2}\) & \(1.23\times 10^{-2}\) & \(7.85\times 10^{-3}\) \\
      & \(N=50\)  & 59  & 52  & \(2.52\times 10^{-4}\) & \(2.12\times 10^{-2}\) & \(1.58\times 10^{-2}\) & \(1.30\times 10^{-2}\) & \(5.67\times 10^{-3}\) & \(1.00\times 10^{-2}\) \\
      & \(N=100\) & 109 & 102 & \(2.38\times 10^{-4}\) & \(1.07\times 10^{-2}\) & \(7.80\times 10^{-3}\) & \(1.06\times 10^{-2}\) & \(4.72\times 10^{-3}\) & \(1.49\times 10^{-2}\) \\
      & \(N=200\) & 209 & 202 & \(\mathbf{2.36\times 10^{-4}}\) & \(\mathbf{5.37\times 10^{-3}}\) & \(\mathbf{3.75\times 10^{-3}}\) & \(\mathbf{1.06\times 10^{-2}}\) & \(\mathbf{4.48\times 10^{-3}}\) & \(2.53\times 10^{-2}\) \\
    \midrule
    \multirow{7}{*}{\shortstack[l]{Convex approx.\\+ MOLP\\(Method~IV)}}
      & \(N=2\)   & 11  & 4   & \(6.37\times 10^{-2}\) & \(3.81\times 10^{-1}\) & --- & \(3.86\) & \(8.65\times 10^{-1}\) & \(1.54\times 10^{-2}\) \\
      & \(N=5\)   & 14  & 7   & \(1.05\times 10^{-2}\) & \(2.07\times 10^{-1}\) & --- & \(6.95\times 10^{-1}\) & \(1.53\times 10^{-1}\) & \(1.90\times 10^{-2}\) \\
      & \(N=10\)  & 19  & 12  & \(1.05\times 10^{-2}\) & \(1.55\times 10^{-1}\) & --- & \(6.94\times 10^{-1}\) & \(9.92\times 10^{-2}\) & \(5.27\times 10^{-2}\) \\
      & \(N=20\)  & 29  & 22  & \(7.95\times 10^{-4}\) & \(5.72\times 10^{-2}\) & --- & \(5.25\times 10^{-2}\) & \(1.37\times 10^{-2}\) & \(1.89\times 10^{-1}\) \\
      & \(N=50\)  & 59  & 52  & \(2.73\times 10^{-4}\) & \(2.84\times 10^{-2}\) & --- & \(1.75\times 10^{-2}\) & \(6.22\times 10^{-3}\) & \(1.26\times 10^{0}\) \\
      & \(N=100\) & 109 & 102 & \(2.73\times 10^{-4}\) & \(2.64\times 10^{-2}\) & --- & \(1.69\times 10^{-2}\) & \(5.06\times 10^{-3}\) & \(6.33\times 10^{0}\) \\
      & \(N=200\) & 209 & 201 & \(2.37\times 10^{-4}\) & \(1.33\times 10^{-2}\) & --- & \(1.06\times 10^{-2}\) & \(4.55\times 10^{-3}\) & \(3.38\times 10^{1}\) \\
    \midrule
    \multirow{6}{*}{\shortstack[l]{MCDPL\\(optimistic)}}
      & \(R=5\)  & --- & 3  & \(5.02\times 10^{-1}\) & \(7.60\times 10^{-1}\) & \(5.90\times 10^{-1}\) & \(4.10\times 10^{1}\) & \(3.82\times 10^{1}\) & \(8.94\times 10^{0}\) \\
      & \(R=10\) & --- & 4  & \(4.48\times 10^{-1}\) & \(7.60\times 10^{-1}\) & \(5.59\times 10^{-1}\) & \(3.98\times 10^{1}\) & \(3.13\times 10^{1}\) & \(1.13\times 10^{1}\) \\
      & \(R=20\) & --- & 4  & \(2.33\times 10^{-1}\) & \(4.51\times 10^{-1}\) & \(3.01\times 10^{-1}\) & \(2.36\times 10^{1}\) & \(1.66\times 10^{1}\) & \(2.30\times 10^{1}\) \\
      & \(R=40\) & --- & 7  & \(2.33\times 10^{-1}\) & \(4.51\times 10^{-1}\) & \(2.76\times 10^{-1}\) & \(2.36\times 10^{1}\) & \(1.33\times 10^{1}\) & \(9.15\times 10^{1}\) \\
      & \(R=60\) & --- & 7  & \(2.33\times 10^{-1}\) & \(4.51\times 10^{-1}\) & \(2.58\times 10^{-1}\) & \(2.36\times 10^{1}\) & \(1.23\times 10^{1}\) & \(2.82\times 10^{2}\) \\
      & \(R=80\) & --- & 14 & \(1.30\times 10^{-1}\) & \(2.44\times 10^{-1}\) & \(1.65\times 10^{-1}\) & \(1.28\times 10^{1}\) & \(7.10\times 10^{0}\) & \(2.45\times 10^{3}\) \\
    \midrule
    \multirow{6}{*}{\shortstack[l]{MCDPL\\(pessimistic)}}
      & \(R=5\)  & --- & 4  & \(1.18\times 10^{-1}\) & \(3.74\times 10^{-1}\) & \(4.93\times 10^{-1}\) & \(8.73\) & \(5.78\) & \(8.94\times 10^{0}\) \\
      & \(R=10\) & --- & 6  & \(1.55\times 10^{-1}\) & \(2.11\times 10^{-1}\) & \(3.72\times 10^{-1}\) & \(1.03\times 10^{1}\) & \(6.20\) & \(1.13\times 10^{1}\) \\
      & \(R=20\) & --- & 9  & \(1.55\times 10^{-1}\) & \(2.11\times 10^{-1}\) & \(3.49\times 10^{-1}\) & \(1.03\times 10^{1}\) & \(6.20\) & \(2.30\times 10^{1}\) \\
      & \(R=40\) & --- & 8  & \(1.18\times 10^{-1}\) & \(1.91\times 10^{-1}\) & \(3.11\times 10^{-1}\) & \(8.73\) & \(5.78\) & \(9.15\times 10^{1}\) \\
      & \(R=60\) & --- & 15 & \(9.35\times 10^{-2}\) & \(1.13\times 10^{-1}\) & \(1.77\times 10^{-1}\) & \(7.96\) & \(3.51\) & \(2.82\times 10^{2}\) \\
      & \(R=80\) & --- & 14 & \(9.35\times 10^{-2}\) & \(1.13\times 10^{-1}\) & \(1.54\times 10^{-1}\) & \(7.96\) & \(2.65\) & \(2.45\times 10^{3}\) \\
    \bottomrule
  \end{tabular*}
  \end{adjustbox}
\end{table*}

\subsection{Discussion}
The benchmarks separate the exact and approximate regimes of the framework.
For polyhedral \glspl{abk:lcdp}, the monolithic lifted formulation is best suited to single-shot queries: it preserves sparsity, avoids projection, and computes exact Pareto frontiers quickly.
Compositional elimination is also exact, but its cost is driven by intermediate projection and redundancy removal, making it most useful when a reduced model is needed for repeated use.
For convex but nonpolyhedral models, polyhedral outer approximations extend the same \gls{abk:molp}-based pipeline beyond the exact linear regime.
The rover benchmark shows that directional inner-front metrics give a clearer view of pessimistic convergence, while accurate~$\left(\varepsilon,0\right)$ outer approximations remain computationally cheap.

\section{Conclusions and Future Work}
\label{sec:conclusions}
This paper introduced \glspl{abk:ldp}, a polyhedral subclass of monotone co-design problems whose feasible functionality--resource relations are described by linear inequalities over Euclidean \glspl{abk:poset}.
We showed that \glspl{abk:ldp} are closed under classic co-design interconnections, so any \gls{abk:lcdp} built from linear components induces a system-level \gls{abk:ldp} whose queries reduce exactly to \glspl{abk:molp}/\glspl{abk:vlp}, yielding a bridge between monotone co-design and polyhedral multiobjective optimization.

We then developed two exact computational pipelines: a \emph{monolithic} formulation that preserves lifted block-angular structure and is effective for single-shot queries, and a \emph{compositional} formulation that eliminates internal variables to produce a reduced external-coordinate model.
Beyond the exact linear case, we showed that convex co-design resource queries admit arbitrarily accurate polyhedral outer approximations; for standard nonnegative resource cones, the recession-cone error is structurally exact ($\delta=0$), so approximation quality is controlled entirely by the bounded-part error~$\varepsilon$.

The numerical results validate both regimes.
On synthetic series-chain \glspl{abk:lcdp}, the monolithic pipeline is fastest, while the compositional pipeline preserves the exact Pareto frontier with lower final query dimension.
On the rigid gripper benchmark, the exact \gls{abk:lcdp} model computes the true frontier with zero approximation error and large runtime gains over \textsc{MCDPL}.
On the rover benchmark, tangent-plane outer approximations produce accurate $\left(\varepsilon,0\right)$ surrogates and a substantially better accuracy--time trade-off than the discretization-based baseline.
Overall, linear co-design emerges as both a rigorous subclass of monotone co-design and a practical backbone for exact and approximate Pareto-set computation.

Future work includes extending the theory to broader convex design classes, exploiting the monolithic block-angular structure through decomposition-based multiobjective algorithms, integrating uncertainty-aware co-design~\cite{huang2025composable, huang2026distributionaluncertaintyadaptivedecisionmaking}, and studying richer interconnection topologies and hybrid monolithic/compositional schemes.

\bibliographystyle{IEEEtran}
\bibliography{paper}
\end{document}